\documentclass[11pt]{amsart}

\usepackage[T1]{fontenc}
\usepackage[utf8]{inputenc}
\usepackage{lmodern}
\usepackage{microtype}
\usepackage{amsmath,amssymb,amsthm,mathtools}
\allowdisplaybreaks
\usepackage[margin=1in]{geometry}
\usepackage{booktabs}
\usepackage{enumitem}
\usepackage[colorlinks=true,linkcolor=blue,citecolor=blue,urlcolor=blue]{hyperref}
\usepackage{url}
\usepackage{aliascnt}

\newtheorem{theorem}{Theorem}[section]

\newaliascnt{proposition}{theorem}
\newtheorem{proposition}[proposition]{Proposition}
\aliascntresetthe{proposition}

\newaliascnt{lemma}{theorem}
\newtheorem{lemma}[lemma]{Lemma}
\aliascntresetthe{lemma}

\newaliascnt{corollary}{theorem}
\newtheorem{corollary}[corollary]{Corollary}
\aliascntresetthe{corollary}

\newaliascnt{conjecture}{theorem}

\aliascntresetthe{conjecture}

\newaliascnt{hypothesis}{theorem}
\newtheorem{hypothesis}[hypothesis]{Hypothesis}
\aliascntresetthe{hypothesis}

\theoremstyle{definition}
\newaliascnt{definition}{theorem}
\newtheorem{definition}[definition]{Definition}
\aliascntresetthe{definition}
\newtheorem*{mainresults}{Summary of main results}

\theoremstyle{remark}
\newaliascnt{remark}{theorem}
\newtheorem{remark}[remark]{Remark}
\aliascntresetthe{remark}

\usepackage[nameinlink,capitalise]{cleveref}
\crefname{theorem}{Theorem}{Theorems}
\Crefname{theorem}{Theorem}{Theorems}
\crefname{proposition}{Proposition}{Propositions}
\Crefname{proposition}{Proposition}{Propositions}
\crefname{lemma}{Lemma}{Lemmas}
\Crefname{lemma}{Lemma}{Lemmas}
\crefname{corollary}{Corollary}{Corollaries}
\Crefname{corollary}{Corollary}{Corollaries}
\crefname{conjecture}{Conjecture}{Conjectures}
\Crefname{conjecture}{Conjecture}{Conjectures}
\crefname{hypothesis}{Hypothesis}{Hypotheses}
\Crefname{hypothesis}{Hypothesis}{Hypotheses}
\crefname{definition}{Definition}{Definitions}
\Crefname{definition}{Definition}{Definitions}
\crefname{remark}{Remark}{Remarks}
\Crefname{remark}{Remark}{Remarks}

\numberwithin{equation}{section}

\newcommand{\N}{\mathbb N}
\newcommand{\Z}{\mathbb Z}
\newcommand{\tauD}{\tau}
\newcommand{\Tmap}{T}
\newcommand{\AnnT}{\mathcal A}
\newcommand{\Frontier}{\mathcal F}
\newcommand{\HST}{\mathrm{H}_{\mathrm{ST}}}
\newcommand{\HQE}{\mathrm{H}_{\mathrm{QE}}}
\newcommand{\EvPlus}{\mathrm{ev+}}
\newcommand{\OddZero}{\mathrm{odd0}}

\newcommand{\DeltaD}{\Delta}
\DeclareMathOperator{\im}{im}

\title[Square-annular dynamics]{Square-Annular Dynamics and Coalescence Frontiers for \(n+\tau(n)\)}

\author[E. Li]{Eric Li}
\dedicatory{\normalfont\normalsize Trinity College, University of Cambridge}
\date{June 13, 2026}
\thanks{Email addresses: \href{mailto:el593@cam.ac.uk}{el593@cam.ac.uk}, \href{mailto:contact@ericli.com}{contact@ericli.com}.}


\hypersetup{
  pdftitle={Square-Annular Dynamics and Coalescence Frontiers for n+tau(n)},
  pdfauthor={Eric Li},
  pdfsubject={Square-annular dynamics and coalescence frontiers for the divisor-successor map},
  pdfkeywords={divisor function, arithmetic dynamics, square annuli, Erdos-Graham problem}
}

\begin{document}
\raggedbottom

\begin{abstract}
Let \(\Tmap(n)=n+\tauD(n)\), where \(\tauD\) is the divisor function.  We study the Erd\H{o}s--Graham coalescence problem by encoding finite-level obstructions in the divisor-successor graph and in square-annular transfer maps.  Coalescence is equivalent both to connectedness of this graph and to synchronization along an infinite non-autonomous sequence of finite annular systems.  The basic identities are
\[
        \im(\AnnT_k)=E_{k+1},
        \qquad
        \Frontier_{k^2}=k^2+E_k,
\]
where \(E_k\) is the set of square-crossing overshoots from below \(k^2\).
We prove a transfer parity law, dynamic frontier bounds for the widths \(W_{k,s}\), and the criterion that \(\liminf_k|\AnnT_k(E_k)|=1\) would imply connectedness.  Unconditionally,
\[
        R(X)\le \log X+2\gamma+O(X^{-1/4}),
\]
and the exit sets are residue-universal, satisfy \(|E_k|\le k^{o(1)}\), and obey
\[
        \frac94K+O(1)\le \sum_{k\le K}|E_k|\ll K(\log K)^3.
\]
Using the shifted-square estimate \(\HST\), obtained from the corrected Henriot--Nair--Tenenbaum theorem in the specialized form of \Cref{prop:henriot-erratum-one-polynomial} and from separate square-shift estimates, we obtain fixed-moment bounds
\[
        \sum_{k\le K}|E_k|^m\ll_m K(\log K)^{C_m}\quad(m\ge2).
\]
A further first-moment refinement to \(K(\log K)^2\) is conditional on the additional, currently unproved, uniform quadratic Euler-product mean-value hypothesis \(\HQE\).  We also prove quantitative large-jump and lower-runner race theorems, isolate interval filling, and formulate a square-gated two-branch criterion.  No proof of the full Erd\H{o}s--Graham problem is claimed.
\end{abstract}

\maketitle

\section{Introduction}

Let
\[
        \tauD(n)=\sum_{d\mid n}1
\]
be the divisor function, and define
\[
        \Tmap(n)=n+\tauD(n) \qquad (n\in\N).
\]
All logarithms are natural; base-two logarithms, when needed, are written explicitly as quotients by \(\log 2\).  Iterates of \(\Tmap\) are always written \(\Tmap^i\), while the annular transfer at level \(k\) is written \(\AnnT_k\).  Erd\H{o}s and Graham \cite[p.~82]{ErdosGraham} ask whether any two positive starting values eventually have a common iterate under \(\Tmap\): for all \(m,n\in\N\), do there exist \(i,j\ge0\) such that
\[
        \Tmap^i(m)=\Tmap^j(n)?
\]
This paper does not solve that problem.  Its aim is to make the finite-level obstructions to coalescence explicit and to give analytic estimates for those obstructions.  Throughout, an unconditional result means unconditional relative to the standard published theorems cited in the bibliography; the only new unproved hypothesis introduced in the paper is \(\HQE\).

Two elementary facts drive the paper.  First, the forward orbits are strictly increasing.  Thus eventual coalescence can be encoded by the connected components of the undirected graph with edges \(n\leftrightarrow \Tmap(n)\).  This graph viewpoint leads to a frontier principle: a component visible below a level \(N\) must cross \(N\), so the number of visible components is bounded by a crossing-endpoint set.  Averaging this crossing count over a short interval and using the Dirichlet divisor problem gives the logarithmic component-count bound in \Cref{thm:component-count}.

Second, the elementary inequality
\[
        \tauD(n)<2\sqrt n
\]
prevents an orbit from skipping a square annulus.  Writing
\[
        I_k=[k^2,(k+1)^2)\cap\N,
\]
each orbit moves from \(I_k\) either within \(I_k\) or into \(I_{k+1}\).  This gives, at each level, a finite transfer map \(\AnnT_k\) from offsets in \(I_k\) to first-entry offsets in \(I_{k+1}\).  The image of this transfer is exactly the exit set \(E_{k+1}\), and the square frontier over \(k^2\) is exactly \(k^2+E_k\).  These identities turn the original infinite coalescence question into a synchronization question for an infinite sequence of finite offset sets, not into a single finite computation.

A point worth emphasizing is that the static exit sets are never singletons: the deficits \(j=1\) and \(j=2\) are active for every \(k\ge2\) and give distinct overshoots.  Consequently a proof cannot come merely from showing that \(|E_k|\) is often small.  The sharper quantities are the dynamic widths
\[
        W_{k,0}=E_k,\qquad
        W_{k,s}=\AnnT_{k+s-1}(W_{k,s-1}),
\]
which measure how much of the square frontier remains after further annular evolution.  In particular, \(\liminf_k|\AnnT_k(E_k)|=1\) would prove connectedness.

The analytic part of the paper estimates the size of \(E_k\).  An elementary root-counting argument gives an unconditional first-moment bound with three logarithms.  For the stronger fixed-moment theory we use the shifted-square estimate \(\HST\), obtained below from the erratum-corrected upper-bound consequence of Henriot's discriminant-uniform Nair--Tenenbaum framework in the precise divisor-power form needed here \cite{NairTenenbaum,Henriot,HenriotErratum}.  The square-shift part is proved here separately from \((X-a)(X+a)\), excluding the zero \(X=a\) and including the short right-hand tail.  The sharper two-logarithm first moment additionally assumes the quadratic Euler-product mean-value input \(\HQE\), isolated explicitly below.  In that refinement, the first-divisor-moment case keeps the good-prime quadratic-character Euler factors through the summation over the deficits \(j\), where their harmonic mean is controlled by \Cref{lem:quadratic-euler-mean}.  The elementary \(K(\log K)^3\) first-moment theorem remains included as an independent unconditional proof that avoids both shifted-square moment technology and real-character Euler-product technology.

\begin{mainresults}
The following list separates the elementary parts from the estimates that use the shifted-square estimate \(\HST\), obtained in \Cref{subsec:analytic-inputs} from the corrected Henriot upper-bound input, and the additional, currently unproved, uniform quadratic Euler-product input \(\HQE\) in \Cref{subsec:analytic-inputs}.  The final square-gated statements are conditional dynamical criteria; no proof of the full Erd\H{o}s--Graham problem is claimed.
\begin{enumerate}[label=\textup{(\roman*)},leftmargin=*]
\item The visible component count satisfies
\[
        R(X)\le \log X+2\gamma+O(X^{-1/4}),
\]
with exponent-only refinements available from sharper divisor-problem estimates.
\item The square-annular transfer maps give an exact finite-level reformulation of the Erd\H{o}s--Graham problem.  The key identities are
\[
        \im(\AnnT_k)=E_{k+1},
        \qquad
        \Frontier_{k^2}=k^2+E_k.
\]
They imply static and dynamic frontier bounds, including \(R(k^2-1)\le |W_{k,s}|\).
\item The exit sets are thin pointwise but arithmetically unavoidable:
\[
        |E_k|\le k^{o(1)},
\]
they are residue-universal, and
\[
        \frac94K+O(1)\le \sum_{2\le k\le K}|E_k|\ll K(\log K)^3.
\]
Using \(\HST\), every fixed moment satisfies
\[
        \sum_{2\le k\le K}|E_k|^m\ll_m K(\log K)^{C_m}\qquad(m\ge2).
\]
Conditional on \(\HST+\HQE\), the first-moment upper bound improves to \(K(\log K)^2\).  Moreover, unconditionally, \(\sum_{2\le k\le K}(|E_k|-2)_+\ge K/4+O(1)\), so a density-one minimal-frontier assertion in this averaged form is impossible.
\item The transfer parity law splits the one-step frontier into two parity families.  This gives a precise obstruction to proving two-branch collapse by moment estimates alone: one needs dynamical same-parity coalescence, not merely small static frontiers.  The later square-gate criterion also exposes a Bunyakovsky-type obstruction to turning two branches into one by purely static divisor estimates.
\item Every orbit has arbitrarily large divisor jumps, and every non-coalescing lower-runner race has unbounded gaps, with explicit CRT block scales.
\end{enumerate}
\end{mainresults}

\begin{center}
\begin{minipage}{0.95\linewidth}
\small
\textbf{Dependency map for the main estimates.}  Here \(\HST\) denotes the shifted-square divisor estimate \Cref{thm:shifted-square-estimate}, obtained from the corrected Henriot upper-bound input in \Cref{prop:henriot-erratum-one-polynomial}, and \(\HQE\) denotes the quadratic Euler-product mean-value hypothesis \Cref{hyp:quadratic-euler-input}.  The square-shift estimates needed alongside the nonsquare Henriot specialization are proved in \Cref{prop:square-shift-estimates}.  Throughout, ``unconditional'' means that no new hypothesis such as \(\HQE\) is assumed; standard cited theorems, including Henriot's corrected Nair--Tenenbaum theorem and the classical divisor estimates, are used as external inputs.
\medskip

\begin{tabular}{@{}p{0.25\linewidth}p{0.68\linewidth}@{}}
\toprule
Method & Results \\
\midrule
Elementary/unconditional & Graph equivalence and adjacent-pair reduction; frontier and component-count estimates \(R(X)\le \log X+2\gamma+O(X^{-1/4})\); square-annular transfer and square-frontier identities; parity law; residue-universality and lower bounds for \(E_k\); elementary average bound \(\sum_{k\le K}|E_k|\ll K(\log K)^3\); CRT large-jump and lower-runner race theorems; interval-filling equivalence; square-gated criterion as a conditional dynamical criterion. \\
Uses corrected Henriot input via \(\HST\) & Shifted-square divisor moments, fixed moments \(\sum_{k\le K}|E_k|^m\ll_m K(\log K)^{C_m}\) for fixed \(m\ge2\), and the corresponding large-value and one-step-width tail bounds. \\
Uses \(\HST+\HQE\) & Conditional only: the two-logarithm first-moment bound \(\sum_{k\le K}|E_k|\ll K(\log K)^2\) and the analogous one-step-width first-moment bound.  These are not unconditional consequences of the present paper. \\
\bottomrule
\end{tabular}
\end{minipage}
\end{center}

The paper is organized as follows.  Sections \ref{sec:graph}--\ref{sec:frontier-count} prove the graph and frontier estimates.  Sections \ref{sec:annuli}--\ref{sec:E-average} develop the square-annular transfer framework and prove extremal, unconditional first-moment, and fixed-moment theorems for the exit sets, and isolate the exact parity obstruction to density-one two-branch collapse.  Sections \ref{sec:races} and \ref{sec:large-jumps} prove large-jump and race-divergence results.  Section \ref{sec:interval-filling} isolates interval filling, and Section \ref{sec:square-gates} records the square-gated two-branch criterion and the Bunyakovsky obstruction to the square-gate route.

\section{The coalescence graph}\label{sec:graph}

Define a relation \(\sim\) on \(\N\) by
\[
        a\sim b
        \quad\Longleftrightarrow\quad
        \exists i,j\ge0\text{ such that }\Tmap^i(a)=\Tmap^j(b).
\]
Let \(\Gamma\) be the undirected graph with vertex set \(\N\) and edge set
\[
        \bigl\{\{n,\Tmap(n)\}:n\in\N\bigr\}.
\]

\begin{lemma}\label{lem:equivalence}
The relation \(\sim\) is an equivalence relation, and its equivalence classes are exactly the connected components of \(\Gamma\).  In particular, the Erd\H{o}s--Graham coalescence problem is equivalent to connectedness of \(\Gamma\).
\end{lemma}

\begin{proof}
Reflexivity and symmetry are clear.  If \(a\sim b\) and \(b\sim c\), choose \(i,j,k,\ell\ge0\) with
\[
        \Tmap^i(a)=\Tmap^j(b),\qquad \Tmap^k(b)=\Tmap^\ell(c).
\]
If \(j\le k\), applying \(\Tmap^{k-j}\) to the first equality gives
\(\Tmap^{i+k-j}(a)=\Tmap^k(b)=\Tmap^\ell(c)\).  If \(k<j\), applying \(\Tmap^{j-k}\) to the second equality gives \(\Tmap^i(a)=\Tmap^j(b)=\Tmap^{\ell+j-k}(c)\).  Thus \(\sim\) is transitive.

Each graph edge joins \(\sim\)-equivalent vertices, so every connected component is contained in a \(\sim\)-class.  Conversely, if \(\Tmap^i(a)=\Tmap^j(b)\), then the two forward paths from \(a\) and \(b\) to this common value form an undirected path from \(a\) to \(b\).  Hence the two notions coincide.
\end{proof}

\begin{lemma}[Adjacent-pair reduction]\label{lem:adjacent}
The graph \(\Gamma\) is connected if and only if \(n\sim n+1\) for every \(n\ge1\).
\end{lemma}

\begin{proof}
If \(\Gamma\) is connected, this is immediate.  Conversely, if every adjacent pair is connected, then
\[
        1\sim2\sim3\sim\cdots\sim n
\]
for every \(n\), by transitivity.  Hence every vertex lies in the component of \(1\).
\end{proof}

\begin{remark}[No point barriers]
A direct one-point barrier cannot work for this map.  For \(N>2\), taking \(m=N-1\) gives
\[
        \Tmap(m)=N-1+\tauD(N-1)\ge N+1,
\]
so no \(N>2\) satisfies \(\Tmap(m)\le N\) for all \(m<N\).  The frontier method below is an interval-free substitute.
\end{remark}

\section{Crossing frontiers and component counts}\label{sec:frontier-count}

For \(N\ge2\), define the crossing endpoint set
\[
        \Frontier_N=\{\Tmap(m):m<N\le \Tmap(m)\}
\]
and the crossing-edge count
\[
        A(N)=\#\{m<N:m+\tauD(m)\ge N\}.
\]
Clearly \(|\Frontier_N|\le A(N)\), since different crossing starts can have the same endpoint.

\begin{lemma}[Frontier lemma]\label{lem:frontier}
Every component of \(\Gamma\) that meets \([1,N-1]\) also meets \(\Frontier_N\).  Consequently, if \(R(X)\) is the number of components of \(\Gamma\) meeting \([1,X]\), then for \(X<N\),
\[
        R(X)\le |\Frontier_N|\le A(N).
\]
\end{lemma}

\begin{proof}
Let \(\mathcal C\) be a component meeting \([1,N-1]\), and choose \(x\in\mathcal C\) with \(x<N\).  The forward orbit of \(x\) is strictly increasing, because \(\tauD(n)\ge1\).  Let \(t\ge1\) be the first index for which \(\Tmap^t(x)\ge N\).  Then \(m=\Tmap^{t-1}(x)<N\le \Tmap(m)\), so \(\Tmap(m)\in\Frontier_N\), and this point lies in the same component as \(x\).  Distinct components are disjoint, giving the displayed inequality.
\end{proof}

Let \(\mathcal C(1)\) be the component containing \(1\).

\begin{corollary}[Frontier certificate]\label{cor:frontier-cert}
If \(\Frontier_N\subseteq\mathcal C(1)\), then \([1,N-1]\subseteq\mathcal C(1)\).  Therefore the coalescence problem is equivalent to the existence of arbitrarily large \(N\) for which
\[
        \Frontier_N\subseteq\mathcal C(1).
\]
\end{corollary}

\begin{proof}
The first statement follows immediately from Lemma \ref{lem:frontier}.  If such \(N\) occur arbitrarily far out, then every fixed positive integer is below one of them and so lies in \(\mathcal C(1)\).  The converse is trivial.
\end{proof}

Let
\[
        S(x)=\sum_{n\le x}\tauD(n),\qquad
        \DeltaD(x)=S(x)-x\log x-(2\gamma-1)x.
\]
We first use only the classical Dirichlet divisor estimate
\begin{equation}\label{eq:delta-classical}
        \DeltaD(x)=O(x^{1/2}).
\end{equation}
For optional refinements, call \((\theta,a)\), with \(0<\theta<1\) and \(a\ge0\), an \emph{admissible pointwise divisor pair} if
\begin{equation}\label{eq:delta-pair}
        \DeltaD(x)=O\bigl(x^\theta(\log x)^a\bigr)
        \qquad (x\to\infty).
\end{equation}
Call \(\theta\) admissible in the exponent-only sense if \(\DeltaD(x)=O_\varepsilon(x^{\theta+\varepsilon})\) for every \(\varepsilon>0\).  Huxley's published work gives the exponent-only admissible value \(\theta=131/416\) for the Dirichlet divisor problem \cite{Huxley}.  We also use the standard maximal-order consequence of Wigert's theorem \cite{Wigert},
\begin{equation}\label{eq:max-x-o}
        \max_{n\le x}\tauD(n)=x^{o(1)}.
\end{equation}

\begin{lemma}[Short average of crossings]\label{lem:average-crossings}
Let \(X\ge2\) and \(H\ge1\) be integers, and set
\[
        D(X)=\max_{1\le m<X}\tauD(m).
\]
Then
\[
        \sum_{N=X+1}^{X+H}A(N)
        \le \sum_{m=X}^{X+H-1}\tauD(m)+D(X)^2.
\]
\end{lemma}

\begin{proof}
For a fixed \(m<X+H\), count the integers \(N\) such that
\[
        X<N\le X+H,
        \qquad
        m<N\le m+\tauD(m).
\]
If \(X\le m<X+H\), the number of such \(N\) is at most \(\tauD(m)\).  If \(m<X\), a contribution is possible only when \(m+\tauD(m)>X\).  Since \(\tauD(m)\le D(X)\), such \(m\) must satisfy \(m>X-D(X)\).  There are fewer than \(D(X)\) such starts, and each contributes at most \(D(X)\) levels.  Their total contribution is at most \(D(X)^2\).
\end{proof}

\begin{theorem}[Component-count bound]\label{thm:component-count}
Unconditionally,
\begin{equation}\label{eq:R-classical}
        R(X)\le \log X+2\gamma+O(X^{-1/4}).
\end{equation}
More generally, if \((\theta,a)\) is an admissible pointwise divisor pair, then
\begin{equation}\label{eq:R-pair}
        R(X)\le \log X+2\gamma+O\!\left(
        X^{-(1-\theta)/2}(\log X)^{a/2}\right).
\end{equation}
If \(\theta\) is admissible in the exponent-only sense, then for every \(\varepsilon>0\),
\begin{equation}\label{eq:R-exponent-only}
        R(X)\le \log X+2\gamma+O_\varepsilon\!\left(
        X^{-(1-\theta)/2+\varepsilon}\right).
\end{equation}
\end{theorem}

\begin{proof}
We prove the pointwise-pair estimate for \(X\) sufficiently large; the unconditional estimate is the special case \((\theta,a)=(1/2,0)\).  Let
\[
        H=\left\lfloor X^{(1+\theta)/2}(\log X)^{a/2}\right\rfloor,
\]
with \(H=\lfloor X^{(1+\theta)/2}\rfloor\) when \(a=0\).  Lemma \ref{lem:frontier} gives \(R(X)\le A(N)\) for every integer \(N\in[X+1,X+H]\), hence
\[
        R(X)\le \frac{1}{H}\sum_{N=X+1}^{X+H}A(N).
\]
By Lemma \ref{lem:average-crossings},
\[
        R(X)
        \le \frac{S(X+H-1)-S(X-1)}{H}+\frac{D(X)^2}{H}.
\]
The maximal-order bound \eqref{eq:max-x-o} gives
\[
        \frac{D(X)^2}{H}=X^{-(1+\theta)/2+o(1)}(\log X)^{-a/2},
\]
which is smaller than the claimed error.  For \(f(x)=x\log x+(2\gamma-1)x\), uniformly for this choice of \(H=o(X)\),
\[
        f(X+H-1)-f(X-1)=H(\log X+2\gamma)+O(H^2/X)+O(\log X).
\]
After division by \(H\), the \(O(\log X)\) term gives \(O(\log X/H)\), and the
floor in the definition of \(H\) changes the estimates only by the same smaller-order
quantity; both are absorbed by the displayed error term below.  The Taylor error contributes
\[
        O(H/X)=O\!\left(X^{-(1-\theta)/2}(\log X)^{a/2}\right),
\]
and the divisor-problem error contributes
\[
        O\!\left(\frac{X^\theta(\log X)^a}{H}\right)
        =O\!\left(X^{-(1-\theta)/2}(\log X)^{a/2}\right).
\]
This proves \eqref{eq:R-pair}, and hence \eqref{eq:R-classical}.  The exponent-only form follows by applying the pointwise-pair estimate with \(\theta+\varepsilon\) and absorbing harmless changes in \(\varepsilon\).
\end{proof}

\begin{remark}[Sharper divisor-problem inputs]\label{rem:divisor-specializations}
Huxley's published exponent \(131/416\) gives
\[
        R(X)\le \log X+2\gamma+O_\varepsilon\!\left(
        X^{-285/832+\varepsilon}\right),
\]
since \((1-131/416)/2=285/832\).  We do not record consequences of unpublished preprint exponents, since they do not affect any result of the paper.  This refinement improves only the vanishing error term.  The qualitative conclusion \(R(X)\le\log X+2\gamma+o(1)\) already follows from the classical estimate.
\end{remark}

\begin{corollary}\label{cor:R-limsup}
Unconditionally,
\[
        \limsup_{X\to\infty}\bigl(R(X)-\log X\bigr)\le 2\gamma .
\]
\end{corollary}

\begin{proof}
This is immediate from \eqref{eq:R-classical}.
\end{proof}

\begin{corollary}\label{cor:root-sparse}
Suppose \(\Gamma\) has infinitely many connected components, and let
\[
        r_1<r_2<r_3<\cdots
\]
be their least elements.  Then
\[
        r_j\ge \exp(j-2\gamma-o(1))=e^{-2\gamma+o(1)}e^j=(e^{-2\gamma}+o(1))e^j.
\]
Equivalently, the least representatives of distinct components are exponentially sparse.
\end{corollary}

\begin{proof}
Since \(R(r_j)\ge j\), Theorem \ref{thm:component-count} gives \(j\le\log r_j+2\gamma+o(1)\), which is equivalent to the claim.
\end{proof}

\section{Square annuli and transfer maps}\label{sec:annuli}

For \(k\ge1\), let
\[
        I_k=[k^2,(k+1)^2)\cap\N.
\]
Thus \(|I_k|=2k+1\), and the offsets in \(I_k\) are \(0,1,\ldots,2k\).

\begin{lemma}[No square-annulus skipping]\label{lem:no-skip}
For every \(n\ge1\),
\[
        \Tmap(n)<(\sqrt n+1)^2.
\]
Consequently, if \(n\in I_k\), then \(\Tmap(n)\in I_k\cup I_{k+1}\).
\end{lemma}

\begin{proof}
Pair each divisor \(d<\sqrt n\) with \(n/d>\sqrt n\).  If \(n\) is not a square, this gives \(\tauD(n)<2\sqrt n\); if \(n\) is a square, the middle divisor \(\sqrt n\) is counted once, and again \(\tauD(n)<2\sqrt n\).  Therefore
\[
        \Tmap(n)=n+\tauD(n)<n+2\sqrt n<(\sqrt n+1)^2.
\]
If \(n<(k+1)^2\), then \((\sqrt n+1)^2<(k+2)^2\), so the next iterate cannot reach \(I_{k+2}\).
\end{proof}

Because \(\Tmap(n)>n\), every orbit is strictly increasing and enters every later square annulus.  This permits a finite first-entry map.

\begin{definition}[Annular transfer map]
For \(k\ge1\) and \(0\le r\le2k\), let \(\AnnT_k(r)\) be the offset with which the orbit starting at \(k^2+r\) first enters \(I_{k+1}\).  Equivalently, if \(s\ge1\) is minimal with \(\Tmap^s(k^2+r)\ge(k+1)^2\), then
\[
        \AnnT_k(r)=\Tmap^s(k^2+r)-(k+1)^2.
\]
By Lemma \ref{lem:no-skip},
\[
        \AnnT_k:\{0,1,\ldots,2k\}\longrightarrow\{0,1,\ldots,2k+2\}.
\]
\end{definition}

\begin{proposition}[Backward recursion for \(\AnnT_k\)]\label{prop:backward-recursion}
Let \(L_k=2k+1\).  For \(0\le r\le 2k\), put
\[
        u(r)=r+\tauD(k^2+r).
\]
Then
\[
        \AnnT_k(r)=
        \begin{cases}
        u(r)-L_k, & u(r)\ge L_k,\\
        \AnnT_k(u(r)), & u(r)<L_k.
        \end{cases}
\]
Since \(u(r)>r\), this determines all values of \(\AnnT_k\) by descending induction on \(r\).
\end{proposition}

\begin{proof}
If \(u(r)\ge L_k\), the first step from \(k^2+r\) reaches \((k+1)^2+(u(r)-L_k)\), so the stated overshoot is the first-entry offset.  If \(u(r)<L_k\), the first step remains in the same annulus and lands at offset \(u(r)\); the eventual first-entry offset is therefore \(\AnnT_k(u(r))\).  The inequality \(u(r)>r\) follows from \(\tauD(k^2+r)\ge1\), so descending induction is valid.
\end{proof}

For a set \(A\) of offsets, write \(\AnnT_k(A)=\{\AnnT_k(r):r\in A\}\).  For \(K\ge1\), define
\[
        S_K(0)=\{0,1,\ldots,2K\},
        \qquad
        S_K(t+1)=\AnnT_{K+t}(S_K(t)).
\]
Thus \(S_K(t)\) is the set of first-entry offsets in \(I_{K+t}\) arising from all starting values in \(I_K\).

\begin{theorem}[Annular synchronization equivalence]\label{thm:annular-equiv}
The Erd\H{o}s--Graham coalescence problem for \(\Tmap(n)=n+\tauD(n)\) is equivalent to the following finite-level synchronization assertion:
\[
        \forall K\ge1\ \exists t\ge0\quad |S_K(t)|=1.
\]
\end{theorem}

\begin{proof}
Let \(b_k\) denote the first-entry offset of the orbit of \(1\) into \(I_k\).  Since every orbit enters every annulus, \(b_k\) is defined for all \(k\ge1\), and \(b_{k+1}=\AnnT_k(b_k)\).

If \(|S_K(t)|=1\), then all starts in \(I_K\) have the same first-entry value in \(I_{K+t}\).  The start in \(I_K\) equal to the first-entry point of the orbit of \(1\) into \(I_K\) is among these starts, so the singleton is necessarily \(\{b_{K+t}\}\).  Hence every start in \(I_K\) meets the orbit of \(1\).  If this holds for every \(K\), then all positive integers lie in the component of \(1\).

Conversely, suppose all positive integers meet the orbit of \(1\).  For a fixed \(K\), the finite set \(I_K\) consists of \(2K+1\) starts, each of which meets the orbit of \(1\) after finitely many steps.  Choose a square annulus beyond all of these meeting points.  At that level all first-entry offsets agree with the main offset, so \(|S_K(t)|=1\) for the corresponding \(t\).
\end{proof}

\section{Exit sets, square frontiers, and dynamic widths}\label{sec:exit-sets}

For \(k\ge2\), define the exit set into \(I_k\) by
\[
        E_k=\{\tauD(k^2-j)-j:
        1\le j\le2k-2,\ \tauD(k^2-j)\ge j\}.
\]
It is convenient to put \(E_1=\{0\}\), corresponding to the fact that the orbit of \(1\) begins at the first point of \(I_1\).  The omitted endpoint \(j=2k-1\) would be the square \((k-1)^2\), and it is never active because \(\tauD((k-1)^2)<2(k-1)<2k-1\).  Thus the definition for \(k\ge2\) records exactly all possible overshoots when crossing the square \(k^2\) from below.

\begin{proposition}[Exact exit-set identity]\label{prop:image-E}
For every \(k\ge1\),
\[
        \im(\AnnT_k)=E_{k+1}.
\]
\end{proposition}

\begin{proof}
If an orbit starting in \(I_k\) first enters \(I_{k+1}\), let its last value below \((k+1)^2\) be \((k+1)^2-j\).  Since this value lies in \(I_k\), we have \(1\le j\le2k+1\).  The endpoint \(j=2k+1\) would be the square \(k^2\), and it is not active because \(\tauD(k^2)<2k<2k+1\).  Thus \(1\le j\le2k=2(k+1)-2\).  The landing offset is
\[
        \Tmap((k+1)^2-j)-(k+1)^2=\tauD((k+1)^2-j)-j,
\]
and crossing means \(\tauD((k+1)^2-j)\ge j\).  Hence the offset lies in \(E_{k+1}\).

Conversely, if \(1\le j\le2k\) and \(\tauD((k+1)^2-j)\ge j\), then the point \((k+1)^2-j\) lies in \(I_k\).  Starting at this point, the next step crosses into \(I_{k+1}\) with offset \(\tauD((k+1)^2-j)-j\).  Thus every element of \(E_{k+1}\) occurs in \(\im(\AnnT_k)\).
\end{proof}

\begin{definition}[Thin annular graph]
The thin annular graph has vertices \((k,r)\) with \(k\ge1\) and \(r\in E_k\), and a directed edge
\[
        (k,r)\longrightarrow (k+1,\AnnT_k(r)).
\]
The main branch is the path \((k,b_k)\), where \(b_k\) is the first-entry offset of the orbit of \(1\) into \(I_k\).
\end{definition}

\begin{proposition}[Bad orbits as thin paths]\label{prop:thin-paths}
A forward orbit meets the orbit of \(1\) if and only if its corresponding thin annular path eventually meets the main branch.  Consequently, failure of global coalescence is equivalent to the existence of an infinite path through the thin annular graph that avoids the main branch at every sufficiently large level.
\end{proposition}

\begin{proof}
Once an orbit has made one square crossing, Proposition \ref{prop:image-E} places every later first-entry offset in the corresponding \(E_k\), and the successive offsets are related by the annular transfer maps.  Thus every actual orbit determines a path in the thin graph after deleting finitely many initial annuli.  Conversely, every thin path starting from a specified vertex \((k,r)\) is realized by the actual forward orbit of the integer \(k^2+r\): the outgoing edge from \((k,r)\) is defined precisely by following that orbit until its first entry into \(I_{k+1}\), and iteration gives the later vertices.

If the first-entry offset of an orbit at level \(k\) equals the main offset \(b_k\), then the two orbits are at the same integer \(k^2+b_k\), and therefore coalesce.  Conversely, if an orbit coalesces with the orbit of \(1\) at some value lying inside an annulus, then after both are followed to the next square annulus their first-entry offsets agree; hence the corresponding thin path meets the main branch.  Therefore a forward orbit meets the orbit of \(1\) exactly when its thin path eventually meets the main branch.

If global coalescence fails, choose an orbit that never meets the orbit of \(1\).  After its first square crossing it yields an infinite thin path avoiding the main branch at every later level.  Conversely, an infinite thin path that avoids the main branch from some level onward is realized by the orbit from its initial vertex, and the preceding paragraph shows that this orbit never coalesces with the orbit of \(1\).
\end{proof}

\begin{proposition}[Square frontier identity]\label{prop:square-frontier}
For every \(k\ge2\),
\[
        \Frontier_{k^2}=k^2+E_k
        :=\{k^2+r:r\in E_k\}.
\]
Consequently, for every \(X<k^2\),
\[
        R(X)\le |E_k|.
\]
\end{proposition}

\begin{proof}
Let \(y\in\Frontier_{k^2}\).  Then \(y=\Tmap(m)\) for some \(m<k^2\le \Tmap(m)\).  Write \(m=k^2-j\).  Since crossing requires \(j\le\tauD(k^2-j)\), and since \(\tauD(k^2-j)<2\sqrt{k^2-j}<2k\), we first get \(1\le j\le2k-1\).  The endpoint \(j=2k-1\) would give \(m=(k-1)^2\), but then \(\tauD(m)<2(k-1)<2k-1\), so it cannot cross.  Hence \(1\le j\le2k-2\), and
\[
        y-k^2=\tauD(k^2-j)-j\in E_k.
\]
Thus \(\Frontier_{k^2}\subseteq k^2+E_k\).  Conversely, if \(r=\tauD(k^2-j)-j\in E_k\), then \(m=k^2-j<k^2\) and \(\Tmap(m)=k^2+r\ge k^2\), so \(k^2+r\in\Frontier_{k^2}\).  This proves the identity.  The inequality \(R(X)\le |E_k|\) follows from Lemma \ref{lem:frontier} with \(N=k^2\).
\end{proof}

\begin{corollary}[Square-frontier certificates]\label{cor:square-frontier-cert}
If \(k^2+E_k\subseteq\mathcal C(1)\), then
\[
        [1,k^2-1]\subseteq\mathcal C(1).
\]
Hence the Erd\H{o}s--Graham problem is equivalent to the existence of arbitrarily large \(k\) for which \(k^2+E_k\subseteq\mathcal C(1)\).
\end{corollary}

\begin{proof}
This is Corollary \ref{cor:frontier-cert} at the square level \(N=k^2\), using Proposition \ref{prop:square-frontier}.  The equivalence follows by letting \(k\to\infty\); the converse is immediate if the graph is connected.
\end{proof}

\begin{corollary}[Static exit-set criterion]\label{cor:liminf-E}
If \(L\) is a nonnegative integer and \(\liminf_{k\to\infty}|E_k|\le L\), then \(\Gamma\) has at most \(L\) connected components.  In particular, if \(\liminf |E_k|=2\), then \(\Gamma\) has at most two connected components.
\end{corollary}

\begin{proof}
For each fixed \(X\), choose arbitrarily large \(k\) with \(X<k^2\) and \(|E_k|\le L\).  Proposition \ref{prop:square-frontier} gives \(R(X)\le L\).  Taking the supremum over \(X\) shows that the total number of components is at most \(L\).
\end{proof}

\begin{definition}[Confluence widths]\label{def:confluence-widths}
For \(k\ge2\), define
\[
        W_{k,0}=E_k,
        \qquad
        W_{k,s}=\AnnT_{k+s-1}(W_{k,s-1})\quad(s\ge1).
\]
Thus \(W_{k,1}=\AnnT_k(E_k)\subseteq E_{k+1}\), and in general \(W_{k,s}\subseteq E_{k+s}\).
\end{definition}

\begin{proposition}[Dynamic frontier bound]\label{prop:dynamic-frontier}
For every \(k\ge2\) and every \(s\ge0\),
\[
        R(k^2-1)\le |W_{k,s}|.
\]
\end{proposition}

\begin{proof}
By Proposition \ref{prop:square-frontier}, every component meeting \([1,k^2-1]\) contains a square-crossing endpoint \(k^2+r\) with \(r\in E_k\).  Starting from that endpoint and following the orbit through the next \(s\) square annuli gives a first-entry offset in \(W_{k,s}\).  If two components land at the same offset in \(W_{k,s}\), then they contain the same integer and are the same component.  Hence the number of components meeting \([1,k^2-1]\) is at most \(|W_{k,s}|\).
\end{proof}

\begin{proposition}[Ultimate confluence widths]\label{prop:ultimate-widths}
For each fixed \(k\ge2\), the sequence \(|W_{k,s}|\) is nonincreasing in \(s\) and therefore stabilizes.  If
\[
        \ell_k=\lim_{s\to\infty}|W_{k,s}|,
\]
then
\[
        \ell_k=R(k^2-1).
\]
Consequently, if the total number of connected components of \(\Gamma\) is finite, then \(\ell_k\) is eventually equal to that total number; if the total number is infinite, then \(\ell_k\to\infty\).  In particular, \(\Gamma\) is connected if and only if \(\ell_k=1\) for arbitrarily large \(k\).
\end{proposition}

\begin{proof}
The monotonicity is immediate from
\[
        W_{k,s+1}=\AnnT_{k+s}(W_{k,s}),
\]
since the image of a finite set has cardinality at most the cardinality of the set.  Hence the nonempty integer sequence \(|W_{k,s}|\) stabilizes.

For \(r\in E_k\), let \(\Psi_s(r)\) be the element of \(W_{k,s}\) obtained by starting at \(k^2+r\) and recording the first-entry offset after \(s\) further square annuli.  If \(\Psi_s(r)=\Psi_s(r')\) for some \(s\), then the two corresponding orbits contain the same first-entry integer in \(I_{k+s}\), so \(k^2+r\) and \(k^2+r'\) lie in the same component of \(\Gamma\).

Conversely, suppose \(k^2+r\) and \(k^2+r'\) lie in the same component.  By Lemma \ref{lem:equivalence}, membership in the same component is the same as eventual coalescence of forward orbits: there are iterates of the two starting values that are equal.  Choose a square annulus strictly beyond that common value.  From that annulus onward the two forward orbits have the same first-entry integer, hence \(\Psi_s(r)=\Psi_s(r')\) for all sufficiently large \(s\).  Since \(E_k\) is finite, there are only finitely many pairs \(r,r'\); taking the maximum of the finitely many required coalescence levels shows that, for all sufficiently large \(s\), the fibres of \(\Psi_s\) are exactly the component classes represented by \(k^2+E_k\).

By the square-frontier identity, every component meeting \([1,k^2-1]\) is represented by at least one point of \(k^2+E_k\), and every point of \(k^2+E_k\) has a predecessor below \(k^2\).  Therefore the number of component classes represented by \(k^2+E_k\) is exactly \(R(k^2-1)\).  This proves \(\ell_k=R(k^2-1)\).  The final assertions follow because \(R(k^2-1)\) increases to the total number of components, in the extended sense.
\end{proof}

\begin{corollary}[Dynamic confluence criterion]\label{cor:dynamic-confluence}
Suppose that for some fixed nonnegative integer \(L\) there are arbitrarily large pairs \((k,s)\), with \(k\to\infty\), such that \(|W_{k,s}|\le L\).  Then \(\Gamma\) has at most \(L\) connected components.  In particular,
\[
        \liminf_{k\to\infty}|\AnnT_k(E_k)|=1
\]
would imply connectedness of \(\Gamma\), while \(\liminf_{k\to\infty}|\AnnT_k(E_k)|\le2\) would imply that \(\Gamma\) has at most two connected components.
\end{corollary}

\begin{proof}
For every fixed \(X\), choose one of the pairs with \(X<k^2\).  Proposition \ref{prop:dynamic-frontier} gives \(R(X)\le L\).  Taking the supremum over \(X\) gives the first assertion.  The special cases use \(s=1\).
\end{proof}

\begin{center}
\begin{minipage}{0.96\linewidth}
\scriptsize
\raggedright
\textbf{Illustrative small square-frontier data only.}  The column \(\AnnT_k(E_k)\) is the one-step dynamic frontier; no proof in the paper relies on this table.
\par\medskip\noindent
\begin{tabular}{@{}c p{0.28\linewidth} p{0.24\linewidth} p{0.24\linewidth}@{}}
\toprule
\(k\) & active deficits \(j\) & \(E_k\) & \(\AnnT_k(E_k)\) \\
\midrule
2 & \(\{1,2\}\) & \(\{0,1\}\) & \(\{0\}\) \\
3 & \(\{1,2,3\}\) & \(\{0,1,3\}\) & \(\{2\}\) \\
4 & \(\{1,2,4\}\) & \(\{2,3\}\) & \(\{0,7\}\) \\
5 & \(\{1,2,3,4,5\}\) & \(\{0,1,7\}\) & \(\{2\}\) \\
6 & \(\{1,2,3,4,6\}\) & \(\{1,2,3\}\) & \(\{1,2\}\) \\
7 & \(\{1,2,3,4,5,7\}\) & \(\{0,1,2,9\}\) & \(\{0,2,5\}\) \\
8 & \(\{1,2,4,8\}\) & \(\{0,2,5,8\}\) & \(\{0,3,5\}\) \\
9 & \(\{1,2,3,4,5,6,9\}\) & \(\{0,1,3,5,9\}\) & \(\{2,8\}\) \\
10 & \(\{1,2,4,10\}\) & \(\{2,4,5,8\}\) & \(\{1,2,15\}\) \\
11 & \(\{1,2,3,4,5,7,9\}\) & \(\{1,2,15\}\) & \(\{0,1,2\}\) \\
12 & \(\{1,2,3,4,6,8,12\}\) & \(\{0,1,2,3,8\}\) & \(\{3,4\}\) \\
\bottomrule
\end{tabular}
\par\medskip\noindent
\emph{Reproducibility note.}  The table is not used in any proof.  It is obtained by the definitions of \(E_k\) and \(\AnnT_k\): enumerate \(1\le j\le2k-2\), retain those with \(\tauD(k^2-j)\ge j\), form the overshoots \(\tauD(k^2-j)-j\), and compute \(\AnnT_k\) by the recursion in \Cref{prop:backward-recursion}.  The computation is finite and exact for each displayed \(k\).
\end{minipage}
\end{center}

\section{A parity law for transfers}\label{sec:parity}

The only square in \(I_k\) is its left endpoint \(k^2\).  Since \(\tauD(n)\) is odd exactly when \(n\) is a square, parity is rigid within an annulus.

\begin{proposition}[Transfer parity law]\label{prop:parity}
For every \(k\ge1\),
\[
        r>0 \quad\Longrightarrow\quad \AnnT_k(r)\equiv r+1\pmod2,
\]
and
\[
        \AnnT_k(0)\equiv0\pmod2.
\]
\end{proposition}

\begin{proof}
If \(r>0\), the orbit segment from \(k^2+r\) until its first entry into \(I_{k+1}\) contains no square before crossing, because the only square in \([k^2,(k+1)^2)\) is \(k^2\).  Hence every increment \(\tauD(n)\) along the segment is even, and the parity of the integer itself is preserved until the landing point.  Therefore, if \(r'=\AnnT_k(r)\), then
\[
        r'\equiv k^2+r-(k+1)^2\equiv r+1\pmod2.
\]
For \(r=0\), the first step starts from the square \(k^2\), so the first increment \(\tauD(k^2)\) is odd.  Also \(\tauD(k^2)<2k\), so this first step does not reach \((k+1)^2\).  After that first step, the orbit is strictly inside \(I_k\) until crossing, and all further increments before crossing are even.  Thus the landing integer has parity \(k^2+1\).  Its offset modulo \((k+1)^2\) is therefore congruent to
\[
        k^2+1-(k+1)^2=-2k\equiv0\pmod2.
\]
\end{proof}

\begin{proposition}[Cross-parity coalescence requires a square]\label{prop:cross-parity-square}
Let \(x,y\in\N\) have opposite parity.  If \(\Tmap^a(x)=\Tmap^b(y)\) for some \(a,b\ge0\), then at least one of the values
\[
        x,\Tmap(x),\ldots,\Tmap^{a-1}(x),\qquad
        y,\Tmap(y),\ldots,\Tmap^{b-1}(y)
\]
is a square, with the empty list omitted if \(a=0\) or \(b=0\).
\end{proposition}

\begin{proof}
The divisor count \(\tauD(n)\) is odd exactly when \(n\) is a square.  Hence \(\Tmap(n)=n+\tauD(n)\) changes the parity of \(n\) exactly at square starting values, and preserves parity at nonsquare starting values.  If neither of the two listed orbit segments contained a square, then \(\Tmap^a(x)\) would have the parity of \(x\) and \(\Tmap^b(y)\) would have the parity of \(y\), contradicting the equality because \(x\) and \(y\) have opposite parity.
\end{proof}

\begin{proposition}[Square-frontier parity split]\label{prop:square-frontier-parity}
Let \(k\ge2\), and let \(j\) be an active deficit for \(E_k\), so \(1\le j\le2k-2\) and \(j\le\tauD(k^2-j)\).  Put
\[
        r=\tauD(k^2-j)-j\in E_k.
\]
Then \(k^2-j\) is not a square, \(\tauD(k^2-j)\) is even, and
\[
        r\equiv j\pmod2.
\]
Equivalently, the square-frontier endpoint \(k^2+r\) has the same parity as the crossing start \(k^2-j\).  In the coordinates of the previous annulus, if \(s=2k-1-j\) is the offset of \(k^2-j\) in \(I_{k-1}\), then
\[
        r\equiv s+1\pmod2.
\]
\end{proposition}

\begin{proof}
The interval \([(k-1)^2,k^2)\) contains only one square, namely \((k-1)^2\), which corresponds to the excluded endpoint \(j=2k-1\).  Thus \(k^2-j\) is nonsquare for every active \(j\), and \(\tauD(k^2-j)\) is even.  Therefore
\[
        r=\tauD(k^2-j)-j\equiv -j\equiv j\pmod2.
\]
The equality of integer parities follows from \(\Tmap(k^2-j)=k^2+r\) and the fact that the increment is even.  Finally, \(s=2k-1-j\), so \(s+1=2k-j\equiv j\pmod2\).
\end{proof}

\section{The size and moments of the exit sets}\label{sec:E-size}

The exit sets \(E_k\) are thin in maximal-order terms, but they are not eventually bounded and they do not eventually avoid any fixed residue class.

\begin{proposition}[Pointwise upper bound]\label{prop:E-upper}
As \(k\to\infty\),
\[
        |E_k|\le
        \exp\left((2\log 2+o(1))\frac{\log k}{\log\log k}\right)=k^{o(1)}.
\]
More precisely, every active deficit \(j\) and every overshoot \(r=\tauD(k^2-j)-j\in E_k\) satisfy
\[
        j,r\le D_k,
        \qquad
        D_k:=\max_{n<k^2}\tauD(n),
\]
and \(D_k\) satisfies the displayed bound.
\end{proposition}

\begin{proof}
If \(r=\tauD(k^2-j)-j\in E_k\), then \(j\le\tauD(k^2-j)\le D_k\), and also \(r\le\tauD(k^2-j)\le D_k\).  Therefore the number of possible active deficits is at most \(D_k\), and \(|E_k|\le D_k\).  Wigert's maximal-order theorem \cite{Wigert} gives
\[
        \max_{n\le x}\tauD(n)
        =\exp\left((\log 2+o(1))\frac{\log x}{\log\log x}\right).
\]
Taking \(x=k^2\) gives the result.
\end{proof}

\begin{theorem}[Residue-universality of exit sets]\label{thm:E-residue-universal}
Fix a modulus \(q\ge1\), a residue class \(a\pmod q\), and an integer \(R\ge1\).  Then there is an arithmetic progression of integers \(k\) such that, for all sufficiently large \(k\) in the progression, \(E_k\) contains at least \(R\) distinct elements congruent to \(a\pmod q\).  In particular,
\[
        \limsup_{k\to\infty}|E_k\cap(a+q\Z)|=\infty.
\]
\end{theorem}

\begin{proof}
Choose a positive integer \(L\), divisible by \(q\) and at least \(2\), so large that there are distinct positive integers
\[
        j_1,\ldots,j_R<L
\]
with \(j_i\equiv -a\pmod q\).  We force \(L\mid \tauD(k^2-j_i)\) for all \(i\), while keeping the resulting overshoots distinct modulo \(L\).

For each \(i\), choose a prime \(p_i>L\), with the \(p_i\)'s pairwise distinct, such that
\[
        p_i\equiv1\pmod {4j_i}.
\]
Dirichlet's theorem on primes in arithmetic progressions supplies infinitely many choices \cite[Chs.~6,~20]{HardyWright}.  The congruence \(p_i\equiv1\pmod {4j_i}\) implies that \(j_i\) is a quadratic residue modulo \(p_i\).  Indeed, \(p_i\equiv1\pmod4\), and for every odd prime \(\ell\mid j_i\) one has \(p_i\equiv1\pmod\ell\), so quadratic reciprocity gives \((\ell/p_i)=(p_i/\ell)=1\).  If \(2\mid j_i\), then \(p_i\equiv1\pmod8\), so \((2/p_i)=1\).  Multiplying the relevant Legendre symbols, with even prime-power exponents contributing trivially, gives \((j_i/p_i)=1\).  Since \(p_i\nmid j_i\), Hensel lifting gives a solution \(u_i\) of
\[
        u_i^2\equiv j_i\pmod {p_i^{L-1}}
\]
with \(p_i\nmid u_i\).  Write \(u_i^2-j_i=c_i p_i^{L-1}\).  Among the classes \(u_i+t p_i^{L-1}\pmod {p_i^L}\), choose \(t\) so that \(c_i+2u_i t\not\equiv0\pmod {p_i}\).  The resulting residue \(b_i\pmod {p_i^L}\) satisfies
\[
        v_{p_i}(b_i^2-j_i)=L-1.
\]
The moduli \(p_i^L\) are pairwise coprime, so the Chinese remainder theorem gives an arithmetic progression of \(k\)'s satisfying
\[
        k\equiv b_i\pmod {p_i^L}\qquad(1\le i\le R).
\]
For every such \(k\),
\[
        v_{p_i}(k^2-j_i)=L-1,
\]
and therefore the factor \(L=(L-1)+1\) divides \(\tauD(k^2-j_i)\).  Taking \(k\) sufficiently large within the progression ensures \(1\le j_i\le2k-2\).  Since \(L>j_i\), each \(j_i\) is active:
\[
        \tauD(k^2-j_i)\ge L>j_i.
\]
Thus \(r_i:=\tauD(k^2-j_i)-j_i\in E_k\).  Moreover \(L\mid\tauD(k^2-j_i)\), so
\[
        r_i\equiv -j_i\pmod L.
\]
The \(j_i\)'s are distinct modulo \(L\), hence the \(r_i\)'s are distinct.  Since \(L\) is divisible by \(q\) and \(j_i\equiv -a\pmod q\), each \(r_i\equiv a\pmod q\).
\end{proof}

\begin{corollary}[Exit sets are unbounded in each parity]\label{cor:E-unbounded-parity}
For every \(R\ge1\) and every \(\varepsilon\in\{0,1\}\), there are infinitely many \(k\) for which \(E_k\) contains at least \(R\) distinct elements congruent to \(\varepsilon\pmod2\).  In particular,
\[
        \limsup_{k\to\infty}|E_k|=\infty.
\]
\end{corollary}

\begin{proof}
Take \(q=2\) and \(a=\varepsilon\) in Theorem \ref{thm:E-residue-universal}.
\end{proof}

For \(k\ge2\), let
\[
        a_k=\#\{1\le j\le 2k-2:\tauD(k^2-j)\ge j\}
\]
denote the number of active deficits at the square \(k^2\).  Thus \(|E_k|\le a_k\), since several active deficits may give the same overshoot.

\begin{proposition}[Exact active-deficit reindexing]\label{prop:active-reindexing}
For an integer \(m\ge1\), let
\[
        g(m)=\bigl(\textup{the least square strictly larger than }m\bigr)-m .
\]
Then, for every \(K\ge2\),
\[
        \sum_{2\le k\le K} a_k
        =\#\{1\le m\le K^2:\tauD(m)\ge g(m)\}.
\]
\end{proposition}

\begin{proof}
If \((k-1)^2<m<k^2\), then the least square strictly larger than \(m\) is \(k^2\), so \(g(m)=k^2-m\).  Writing \(j=k^2-m\) gives exactly the range \(1\le j\le2k-2\) used in the definition of \(a_k\).  Thus non-square \(m\) in this interval contributes to the right-hand side exactly when its corresponding deficit is active at level \(k\).

It remains only to check the square endpoints, which are harmless.  If \(m=\ell^2\), then \(g(m)=(\ell+1)^2-\ell^2=2\ell+1\), while \(\tauD(\ell^2)<2\ell<2\ell+1\).  Hence squares never contribute to the right-hand side.  This includes \(m=K^2\), so the displayed identity is exact.
\end{proof}

\begin{proposition}[Baseline and explicit lower-density lower bounds]\label{prop:E-lower-bound}
For every \(k\ge2\),
\[
        |E_k|\ge2.
\]
Moreover, \(a_k\ge3\) for every \(k\ge3\), and every
\[
        k\equiv3,5\pmod 8
\]
has \(|E_k|\ge3\).  Consequently, as \(K\to\infty\),
\[
        \sum_{2\le k\le K}a_k\ge 3K+O(1),
        \qquad
        \sum_{2\le k\le K}|E_k|\ge \frac94K+O(1).
\]
More generally, for every \(R\ge1\) there is a positive-density arithmetic progression of levels \(k\) for which \(|E_k|\ge R\) for all sufficiently large \(k\) in that progression.
\end{proposition}

\begin{proof}
The deficits \(j=1\) and \(j=2\) are active for every \(k\ge2\), because \(\tauD(k^2-1)\ge2\) and \(\tauD(k^2-2)\ge2\).  The integers \(k^2-1\) and \(k^2-2\) are not squares, so their divisor counts are even.  Hence the two corresponding overshoots
\[
        r_1=\tauD(k^2-1)-1,
        \qquad
        r_2=\tauD(k^2-2)-2
\]
have opposite parity and are distinct.  This proves \(|E_k|\ge2\).

If \(k\ge3\) is odd, then \(k^2-3\) is an even integer larger than \(2\).  It is not a square: if \(k^2-3=s^2\), then \((k-s)(k+s)=3\), which has no solution with odd \(k\ge3\).  Hence it is composite, and since it is not a square its divisor count is even and at least \(4\).  Thus \(j=3\) is active for every odd \(k\ge3\).  If \(k\ge4\) is even, then
\[
        k^2-4=(k-2)(k+2)
\]
is a positive composite integer.  It is not a square: if \(k^2-4=s^2\), then \((k-s)(k+s)=4\), which has no solution with \(k\ge4\).  Therefore \(\tauD(k^2-4)\ge4\).  Thus \(j=4\) is active for every even \(k\ge4\).  Together with the always-active deficits \(j=1,2\), this proves \(a_k\ge3\) for every \(k\ge3\).

We next show that, on the residue classes \(k\equiv3,5\pmod8\), the third active deficit gives a genuinely new exit offset.  Put
\[
        r_3=\tauD(k^2-3)-3.
\]
The parity of \(r_3\) is odd, whereas \(r_2\) is even, so \(r_3\ne r_2\).  It remains to rule out \(r_3=r_1\).  If \(k\equiv3,5\pmod8\), then \(v_2(k^2-1)=3\), and hence \(4\mid\tauD(k^2-1)\).  On the other hand, for every odd \(k\),
\[
        k^2-3=2m,
        \qquad
        m=\frac{k^2-3}{2}\equiv3\pmod4.
\]
Thus \(m\) is not a square, so \(\tauD(m)\) is even and \(4\mid\tauD(k^2-3)=2\tauD(m)\).  If \(r_3=r_1\), then
\[
        \tauD(k^2-3)=\tauD(k^2-1)+2,
\]
which is impossible modulo \(4\).  Therefore \(r_1,r_2,r_3\) are three distinct elements of \(E_k\) for every \(k\equiv3,5\pmod8\).

There are \(K+O(1)\) integers \(k\ge3\) and \(K/4+O(1)\) integers congruent to \(3\) or \(5\pmod8\) up to \(K\).  Combining \(a_k\ge3\) for all \(k\ge3\), the baseline \(|E_k|\ge2\), and the preceding paragraph gives the two displayed sums.  Finally, the residue-universality theorem applied with \(q=1\) gives, for each fixed \(R\), an arithmetic progression of levels on which \(|E_k|\ge R\) for all sufficiently large levels in the progression.
\end{proof}

\begin{corollary}[Linear excess over the minimal frontier]\label{cor:excess-not-o}
The minimal value \(|E_k|=2\) is not attained on a density-one set in the strong averaged sense.  More precisely,
\[
        \sum_{2\le k\le K} (|E_k|-2)_+ \ge \frac14K+O(1).
\]
Consequently the possible strengthening
\[
        \sum_{2\le k\le K}(|E_k|-2)=o(K)
\]
is false.
\end{corollary}

\begin{proof}
For every \(k\equiv3,5\pmod8\), Proposition \ref{prop:E-lower-bound} gives \(|E_k|\ge3\).  These levels have density \(1/4\), and each contributes at least one to \((|E_k|-2)_+\).
\end{proof}

\section{Average and fixed-moment bounds}\label{sec:E-average}

Recall that \(a_k\) denotes the number of active deficits at the square \(k^2\).  Then \(|E_k|\le a_k\), since several active deficits may give the same overshoot.

\begin{lemma}[Roots of a quadratic congruence]\label{lem:root-bound}
For \(j,d\ge1\), put
\[
        \rho_j(d)=\#\{x\pmod d:x^2\equiv j\pmod d\}.
\]
Then
\[
        \rho_j(d)\ll 2^{\omega(d)}(j,d)^{1/2},
\]
with an absolute implied constant.
\end{lemma}

\begin{proof}
By the Chinese remainder theorem it suffices to prove the corresponding bound for prime powers.  Let \(p^\alpha\|d\), and write \(\beta=v_p(j)\), with the convention that \(\beta\ge\alpha\) if \(p^\alpha\mid j\).  If \(\beta\ge\alpha\), then the congruence is \(x^2\equiv0\pmod {p^\alpha}\), which has \(p^{\lfloor \alpha/2\rfloor}\le p^{\min(\alpha,\beta)/2}\) solutions.  If \(\beta<\alpha\) is odd, there are no solutions.  If \(\beta=2c<\alpha\), then any solution has \(x=p^c y\), with \(p\nmid y\), and after division by \(p^{2c}\) one obtains a quadratic congruence for the unit \(y\) modulo \(p^{\alpha-2c}\).  Such a unit congruence has at most two solutions for odd \(p\), and at most four solutions for \(p=2\).  Thus the number of solutions modulo \(p^\alpha\) is \(O(p^c)=O(p^{\min(\alpha,\beta)/2})\), with at most a fixed factor, and with at most a factor \(2\) for each odd prime divisor of \(d\).  Multiplying over prime powers gives the stated estimate.
\end{proof}

For \(j\ge1\), define
\[
        B(j)=\sum_{e\mid j}\frac{2^{\omega(e)}}{\sqrt e}.
\]

\begin{lemma}[A uniform shifted-square divisor sum]\label{lem:shifted-square-divisor-sum}
Uniformly for \(K\ge2\) and \(1\le j\le 2K\),
\[
        \sum_{\substack{k\le K\\ k^2>j}} \tauD(k^2-j)
        \ll K(\log K)^2 B(j).
\]
\end{lemma}

\begin{proof}
Since \(j>0\), we have \(k^2-j<k^2\).  Hence, by divisor pairing,
\[
        \tauD(k^2-j)
        \le 2\sum_{\substack{d<k\\ d\mid k^2-j}}1 .
\]
Therefore
\[
 \sum_{\substack{k\le K\\ k^2>j}} \tauD(k^2-j)
 \le 2\sum_{d\le K}\#\{k\le K:k^2\equiv j\pmod d\}.
\]
For each modulus \(d\), the number of such \(k\) is at most \((K/d+1)\rho_j(d)\).  By Lemma \ref{lem:root-bound}, it remains to bound
\[
        K\sum_{d\le K}\frac{2^{\omega(d)}(j,d)^{1/2}}{d}
        +\sum_{d\le K}2^{\omega(d)}(j,d)^{1/2}.
\]
Using \((j,d)^{1/2}\le\sum_{e\mid(j,d)}\sqrt e\) and \(2^{\omega(em)}\le2^{\omega(e)}2^{\omega(m)}\), the first sum is at most
\[
        \sum_{e\mid j}\frac{2^{\omega(e)}}{\sqrt e}
        \sum_{m\le K/e}\frac{2^{\omega(m)}}{m}
        \ll B(j)(\log K)^2,
\]
because \(2^{\omega(m)}\le\tauD(m)\) and \(\sum_{m\le x}\tauD(m)/m\ll(\log x)^2\).  Similarly, the second sum is at most
\[
        \sum_{e\mid j}\sqrt e\,2^{\omega(e)}
        \sum_{m\le K/e}2^{\omega(m)}
        \ll K(\log K)B(j),
\]
because \(\sum_{m\le x}2^{\omega(m)}\le\sum_{m\le x}\tauD(m)\ll x\log x\).  Combining these estimates gives the claim.
\end{proof}

\begin{theorem}[Elementary average exit-set bound]\label{thm:E-average-elementary}
As \(K\to\infty\),
\[
        \sum_{2\le k\le K}|E_k|
        \le \sum_{2\le k\le K}a_k
        \ll K(\log K)^3 .
\]
\end{theorem}

\begin{proof}
For every active deficit \(j\) at level \(k\), one has \(j\le\tauD(k^2-j)\).  Hence
\[
        a_k\le \sum_{j=1}^{2k-2}\frac{\tauD(k^2-j)}j.
\]
After summing over \(k\le K\) and interchanging the order of summation,
\[
        \sum_{2\le k\le K}a_k
        \le \sum_{j\le2K}\frac1j
        \sum_{\substack{k\le K\\ k^2>j}}\tauD(k^2-j).
\]
By \Cref{lem:shifted-square-divisor-sum}, this is at most
\[
        K(\log K)^2\sum_{j\le2K}\frac{B(j)}j .
\]
Since
\[
\begin{aligned}
        \sum_{j\le2K}\frac{B(j)}j
        &=\sum_{j\le2K}\frac1j\sum_{e\mid j}\frac{2^{\omega(e)}}{\sqrt e}  \\
        &=\sum_{e\le2K}\frac{2^{\omega(e)}}{e^{3/2}}
          \sum_{m\le2K/e}\frac1m
        \ll \log K\sum_{e=1}^{\infty}\frac{2^{\omega(e)}}{e^{3/2}}
        \ll \log K,
\end{aligned}
\]
the active-deficit bound follows.  The estimate for \(|E_k|\) follows from
\(|E_k|\le a_k\).
\end{proof}

\subsection{Shifted-square theorem and Euler-product input}\label{subsec:analytic-inputs}

The unconditional first moment has already been proved by elementary root counting.
In this subsection we record and specialize the external shifted-square divisor input
\(\HST\).  The nonsquare part is a consequence of Henriot's discriminant-uniform
Nair--Tenenbaum theorem in its erratum-corrected upper-bound form; the square shifts are
handled separately in \Cref{prop:square-shift-estimates}.  Thus \(\HST\) is
unconditional relative to the cited corrected theorem, but it is not an elementary result
proved independently here.  We then isolate, as a separate input, the quadratic
Euler-product mean-value statement \(\HQE\).  This input is not proved here.  All
constants below are independent of the shift, the dyadic parameter, the truncation
parameter, and the outer summation parameter; they may depend only on the fixed exponents
named in the statement.

For a real quadratic character \(\chi\) and \(M\ge2\), put
\[
        L_M(1,\chi)=\prod_{p\le M}\left(1-\frac{\chi(p)}p\right)^{-1}.
\]
For \(C\ge0\), put
\[
        \mathfrak D_C(n)=\prod_{p\mid 2n}\left(1+\frac1p\right)^C.
\]
If \(j\) is not a square, \(\chi_j\) denotes the primitive quadratic character attached
to the squarefree kernel of \(j\); for primes \(p\nmid 2j\) this means
\(\chi_j(p)=(j/p)\).

The statement below is not the uncorrected 2012 Corollary 2.  It is the one-polynomial
divisor-power specialization after making the corrections from Henriot's erratum:
replacing \(\rho_{\mathbf R}\) by the corrected function \(\breve\rho_{\mathbf R}\), using
the modulus involving the squarefree kernels \(\kappa(a_i)\), incorporating the
submultiplicativity repair, and replacing \(D^*\) by \(a^*D^*\) in the exceptional-prime
set \cite{Henriot,HenriotErratum}.  In the one-polynomial case there is no
cross-factor congruence obstruction, but the normalized corrected congruence count must
still be kept.

For the present application the nonsquare polynomial is
\[
        Q_j(X)=X^2-j .
\]
When \(j\) is not a square, \(Q_j\) is primitive, irreducible over \(\mathbb Z\), monic,
of degree \(2\), and has no fixed prime divisor.  Its leading coefficient is \(a^*=1\)
and its discriminant is \(D^*=4j\).  Hence the corrected exceptional-prime set is exactly
\(p\mid 2j\).  Since there is only one irreducible factor, the property
\((P_{a_1,\ldots,a_r,n})\) in the erratum has no cross-factor part.  For \(r=1\), it
reduces to the ordinary divisibility condition, with the corrected modulus
\(a\kappa(a)\).  Thus the raw comparison \(\breve\rho_Q(a)\le \rho_Q(a)\) is not the
right assertion; the comparison used below is the normalized one
\[
        \frac{\breve\rho_Q(a)}{a\kappa(a)}\le \frac{\rho_Q(a)}{a},
        \qquad
        \rho_Q(a)=\#\{x\bmod a:Q(x)\equiv0\pmod a\}.
\]
Indeed, every corrected residue class modulo \(a\kappa(a)\) projects to a root modulo
\(a\), and each root modulo \(a\) has at most \(\kappa(a)\) lifts.  For \(a=p^\nu\), this
gives
\[
        \frac{\breve\rho_Q(p^\nu)}{p^\nu\kappa(p^\nu)}
        =\frac{\breve\rho_Q(p^\nu)}{p^{\nu+1}}
        \le \frac{\rho_Q(p^\nu)}{p^\nu}.
\]

\begin{proposition}[Corrected Henriot upper-bound input, one-polynomial divisor-power form]
\label{prop:henriot-erratum-one-polynomial}
Let \(B\ge1\) be fixed.  Let \(Q\in\mathbb Z[X]\) be primitive and irreducible,
of degree \(g\), with no fixed prime divisor, leading coefficient \(a^*\), and corrected
discriminant parameter \(D^*\).  Put
\[
        \rho_Q(p)=\#\{x\bmod p:Q(x)\equiv0\pmod p\}.
\]
Uniformly for fixed \(0<\alpha,\delta\le1\), for
\(x\ge C_0\|Q\|^\delta\) and \(x^\alpha<y\le x\),
\begin{equation}\label{eq:henriot-shiu-one-variable}
        \sum_{x<n\le x+y}\tauD(|Q(n)|)^B
        \ll_{B,g,\alpha,\delta}
        y\prod_{g<p\le x}\left(1-\frac{\rho_Q(p)}p\right)
        \mathfrak B_B(Q;x),
\end{equation}
where
\[
        \mathfrak B_B(Q;x)
        =\sum_{a\le x}\tauD(a)^B\frac{\breve\rho_Q(a)}{a\kappa(a)}.
\]
Here \(\breve\rho_Q\) is the corrected one-polynomial congruence-counting function from
Henriot's erratum.  The constants depend only on \(B,g,\alpha,\delta\) and on the fixed
parameters in Henriot's theorem, not on \(Q,x,y\).
\end{proposition}

\begin{proof}
This is the erratum-corrected one-polynomial upper-bound specialization of Henriot's
Corollary 2, obtained from New Theorem 5 after replacing \(\rho_{\mathbf R}\) by
\(\breve\rho_{\mathbf R}\), using the corrected modulus containing \(\kappa\), and replacing
\(D^*\) by \(a^*D^*\).  Applied with \(k=r=1\) and
\[
        F(n)=\tauD(n)^B,
\]
it gives \eqref{eq:henriot-shiu-one-variable}.  For each fixed real \(B\ge1\), this
function is nonnegative and multiplicative and satisfies the one-variable growth
requirements in Henriot's class \(\mathcal M_1\): for every fixed \(\varepsilon>0\),
\[
        \tauD(n)^B\ll_{B,\varepsilon} n^\varepsilon,
        \qquad
        F(p^\nu)=(\nu+1)^B\ll_{B,\varepsilon}(1+\varepsilon)^\nu .
\]
These are precisely the growth requirements used in the one-variable upper-bound form; no
lower-bound or non-vanishing hypothesis from New Theorem 6 is used.
\end{proof}

\begin{lemma}[Euler-product majorant for \(X^2-j\)]\label{lem:henriot-euler-majorant}
Let \(j\) be nonsquare and \(Q_j(X)=X^2-j\).  For every fixed real \(B\ge1\) there are
constants \(C_B,C_B'\), depending only on \(B\), such that uniformly for \(x\ge2\),
\[
        \mathfrak B_B(Q_j;x)\ll_B
        \mathfrak D_{C_B'}(j)(\log(2x))^{C_B}.
\]
In particular, no factor depending on the prime-power valuations \(v_p(j)\) is needed.
\end{lemma}

\begin{proof}
The function \(\breve\rho_Q\) is multiplicative in the corrected theorem, so the truncated
sum defining \(\mathfrak B_B\) is bounded by the corresponding Euler product over
\(p\le x\).  At every prime we use only the normalized comparison
\[
        \frac{\breve\rho_{Q_j}(p^\nu)}{p^{\nu+1}}
        \le \frac{\rho_{Q_j}(p^\nu)}{p^\nu}.
\]
If \(p\nmid2j\), all roots modulo \(p\) are simple, and Hensel lifting gives
\(\rho_{Q_j}(p^\nu)=\rho_{Q_j}(p)\le2\) for every \(\nu\ge1\).  Hence the good-prime local
factor is
\[
        1+\rho_{Q_j}(p)\sum_{\nu\ge1}\frac{(\nu+1)^B}{p^\nu}
        \le 1+C_B\frac{\rho_{Q_j}(p)}p,
\]
and the product of these factors over good primes is \(\ll_B(\log(2x))^{C_B}\), since
\(\rho_{Q_j}(p)\le2\) and \(\sum_{p\le x}p^{-1}\ll\log\log(2x)\).

It remains to bound the local factors at primes \(p\mid2j\).  The prime \(2\) contributes
only a constant depending on \(B\).  Let \(p\) be an odd prime dividing \(j\), and write
\(\beta=v_p(j)\).  If \(x^2\equiv j\pmod {p^\nu}\), then either \(\beta\ge\nu\), in which
case the number of solutions is at most \(p^{\lfloor\nu/2\rfloor}\), or
\(\beta=2c<\nu\), in which case \(x=p^c y\) and the remaining unit quadratic congruence
has at most two root classes, giving at most \(2p^c\) solutions modulo \(p^\nu\); for odd
\(\beta<\nu\) there are no solutions.  Therefore
\[
        \frac{\rho_{Q_j}(p^\nu)}{p^\nu}\ll p^{-\lceil\nu/2\rceil},
\]
and the bad-prime local factor is at most \(1+C_B/p\).  Multiplying over \(p\mid2j\)
gives the factor \(\mathfrak D_{C_B'}(j)\), after increasing \(C_B'\) if necessary.
\end{proof}

We apply \Cref{prop:henriot-erratum-one-polynomial} below and then restrict to the positive subrange \(Q_j(k)>0\).

\begin{lemma}[Visible first moment for nonsquare shifts]\label{lem:visible-first-nonsquare}
Uniformly for nonsquare \(j\ge1\), \(M\ge2\), and intervals \(M<k\le2M\),
\[
        \sum_{\substack{M<k\le2M\\ k^2>j}}
        \tauD(k^2-j)
        \ll M\log(2M)\,\mathfrak D_C(j)L_{2M}(1,\chi_j),
\]
with an absolute constant \(C\).
\end{lemma}

\begin{proof}
For \(k\le2M\) and \(k^2>j\), divisor pairing gives
\[
        \tauD(k^2-j)
        \le 2\sum_{\substack{d\le2M\\d\mid k^2-j}}1.
\]
Hence, with \(\rho_j(d)=\#\{x\bmod d:x^2\equiv j\pmod d\}\),
\[
\begin{aligned}
        \sum_{\substack{M<k\le2M\\ k^2>j}}\tauD(k^2-j)
        &\ll \sum_{d\le2M}\left(\frac Md+1\right)\rho_j(d)        \\
        &\ll M\sum_{d\le2M}\frac{\rho_j(d)}d,
\end{aligned}
\]
because \(d\le2M\) implies \(M/d+1\ll M/d\).

The function \(\rho_j(d)\) is multiplicative in \(d\).  If \(p\nmid2j\), Hensel lifting gives
\(\rho_j(p^\nu)=1+\chi_j(p)\) for every \(\nu\ge1\), whence
\[
        1+\sum_{\nu\ge1}\frac{\rho_j(p^\nu)}{p^\nu}
        =\frac{1+\chi_j(p)/p}{1-1/p}.
\]
For an odd prime \(p\mid j\), write \(\beta=v_p(j)\).  If \(x^2\equiv j\pmod {p^\nu}\), then either \(\beta\ge\nu\), in which case the number of solutions is at most \(p^{\lfloor\nu/2\rfloor}\), or \(\beta=2c<\nu\), in which case \(x=p^c y\) and the remaining unit quadratic congruence has at most two root classes, giving at most \(2p^c\) solutions modulo \(p^\nu\); for odd \(\beta<\nu\) there are no solutions.  In all cases
\[
        \frac{\rho_j(p^\nu)}{p^\nu}\ll p^{-\lceil\nu/2\rceil},
        \qquad
        1+\sum_{\nu\ge1}\frac{\rho_j(p^\nu)}{p^\nu}\le 1+\frac{C}{p}.
\]
After increasing \(C\), the prime \(2\) is absorbed into the same fixed constant.
Therefore
\[
\begin{aligned}
        \sum_{d\le2M}\frac{\rho_j(d)}d
        &\le \prod_{p\le2M}\left(1+\sum_{\nu\ge1}\frac{\rho_j(p^\nu)}{p^\nu}\right) \\
        &\ll \mathfrak D_C(j)
        \prod_{\substack{p\le2M\\p\nmid2j}}
        \frac{1+\chi_j(p)/p}{1-1/p}.
\end{aligned}
\]
For \(\chi_j(p)=\pm1\),
\[
        1+\frac{\chi_j(p)}p
        =\left(1-\frac1{p^2}\right)
         \left(1-\frac{\chi_j(p)}p\right)^{-1}.
\]
Mertens' theorem and convergence of \(
\prod_p(1-p^{-2})\) now give
\[
        \sum_{d\le2M}\frac{\rho_j(d)}d
        \ll \mathfrak D_C(j)\log(2M)L_{2M}(1,\chi_j).
\]
This proves the lemma.
\end{proof}

\begin{theorem}[Shifted-square divisor estimate \(\HST\)]\label{thm:shifted-square-estimate}
Let \(B\ge1\) be fixed.  There are constants \(C_B,C_B'\), depending only on the fixed real exponent \(B\),
with the following properties.
\begin{enumerate}[label=\textup{(\alph*)},leftmargin=*]
\item If \(j\) is not a square, \(M\ge2\), and the dyadic interval
\(M<k\le2M\) contains a point with \(k^2>j\), then
\[
        \sum_{\substack{M<k\le2M\\ k^2>j}}
        \tauD(k^2-j)^B
        \ll_B M(\log(2M))^{C_B}\mathfrak D_{C_B'}(j).
\]
\item In the case \(B=1\), the good-prime Euler factor may be kept visible:
\[
        \sum_{\substack{M<k\le2M\\ k^2>j}}
        \tauD(k^2-j)
        \ll M\log(2M)\mathfrak D_{C_1'}(j)L_{2M}(1,\chi_j).
\]
\end{enumerate}
\end{theorem}

\begin{proof}
Part (b) is \Cref{lem:visible-first-nonsquare}.  For part (a), put
\(Q_j(X)=X^2-j\).  Because \(j\) is nonsquare, \(Q_j\) is primitive, monic and irreducible over
\(\mathbb Z\); it has degree \(2\), leading coefficient \(1\), and discriminant \(4j\).  It also has no fixed prime
divisor: for an odd prime \(p\), the values \(-j\) and \(1-j\) cannot both vanish
modulo \(p\), and for \(p=2\) one of \(Q_j(0)\) and \(Q_j(1)\) is odd.

Apply \Cref{prop:henriot-erratum-one-polynomial} to \(Q_j\) and the fixed exponent \(B\), with \(x=M\), \(y=M\), fixed
\(\alpha=1/2\), and fixed \(\delta=1/4\).  The condition \(y>x^\alpha\) is immediate.
If the interval contains a point with \(k^2>j\), then \(j<4M^2\), while
\(\|Q_j\|=j+1\).  Hence \(M\ge C_0\|Q_j\|^\delta\) holds for all sufficiently large \(M\)
with this fixed \(\delta\), once \(M\) exceeds a constant depending only on the fixed
Henriot parameters.  If \(M\) is below that constant, then \(j<4M^2\) leaves only
boundedly many possible shifts, and their contribution is absorbed into the implied
constant.  The imported theorem is applied to
\(\tauD(|Q_j(k)|)^B\), and restricting to the subrange \(k^2>j\) only decreases the sum.

By \Cref{lem:henriot-euler-majorant},
\(\mathfrak B_B(Q_j;M)\ll_B\mathfrak D_{C_B'}(j)(\log(2M))^{C_B}\).  The product
\(\prod_{2<p\le M}(1-\rho_j(p)/p)\) is nonnegative and at most \(1\).  Substitution in \eqref{eq:henriot-shiu-one-variable} proves part (a), after enlarging the logarithmic exponent \(C_B\) if necessary.
\end{proof}

\begin{lemma}[Standard divisor moments]\label{lem:standard-divisor-moments}
For every fixed positive integer \(A\),
\[
        \sum_{n\le X}\tauD(n)^A \ll_A X(\log(2X))^{2^A-1}.
\]
Consequently, for every fixed real \(A\ge1\), there is a constant \(C_A\) such that
\[
        \sum_{n\le X}\tauD(n)^A \ll_A X(\log(2X))^{C_A}.
\]
\end{lemma}

\begin{proof}
For integer \(A\), this is the standard Selberg--Delange estimate for divisor moments; see, for example, \cite[Ch. II.5]{Tenenbaum}.  The weaker polylogarithmic form also follows from standard mean-value estimates for multiplicative functions.  For real \(A\), replace \(A\) by \(\lceil A\rceil\), since
\(\tauD(n)^A\le \tauD(n)^{\lceil A\rceil}\).
\end{proof}

\begin{proposition}[Square-shift estimates]\label{prop:square-shift-estimates}
For every fixed real \(B\ge1\) there are constants \(C_B,C_B'\), depending only on \(B\),
such that for every \(a\ge1\) and \(M\ge2\),
\[
        \sum_{\substack{M<k\le2M\\ k>a}}
        \tauD(k^2-a^2)^B
        \ll_B M(\log(2M))^{C_B}\mathfrak D_{C_B'}(a).
\]
For \(B=1\) one has the sharper form
\[
        \sum_{\substack{M<k\le2M\\ k>a}}
        \tauD(k^2-a^2)
        \ll M(\log(2M))^2\mathfrak D_{C_1'}(a).
\]
The estimates include the short right-hand tail \(M<a<k\le2M\), and the zero
\(k=a\) is excluded by the condition \(k>a\).
\end{proposition}

\begin{proof}
If \(a\ge2M\), the summation range is empty.  Hence we may assume \(a<2M\).  Write
\[
        k^2-a^2=(k-a)(k+a)=h(h+2a),\qquad h=k-a\ge1.
\]
For \(M<k\le2M\) and \(k>a\), we have \(1\le h\le2M-a<2M\) and \(h+2a\le6M\).
For fixed real \(B\ge1\),
\[
        \tauD(h(h+2a))^B\le \tauD(h)^B\tauD(h+2a)^B.
\]
By Cauchy's inequality and \Cref{lem:standard-divisor-moments}, we get, uniformly in \(a<2M\),
\[
\begin{aligned}
        \sum_{\substack{M<k\le2M\\ k>a}}\tauD(k^2-a^2)^B
        &\le \sum_{1\le h\le2M}\tauD(h)^B\tauD(h+2a)^B  \\
        &\ll_B M(\log(2M))^{C_B}.
\end{aligned}
\]
This is stronger than the displayed fixed-moment estimate, because
\(\mathfrak D_{C_B'}(a)\ge1\).

For the sharper first-moment estimate, use the following crude additive-divisor upper bound, sufficient for the present application and obtainable by the elementary divisor-pairing argument below:
\begin{equation}\label{eq:additive-divisor-bound}
        \sum_{h\le X}\tauD(h)\tauD(h+\ell)
        \ll X(\log(2X))^2\prod_{p\mid\ell}\left(1+\frac1p\right)^C
        \qquad(1\le\ell\le2X),
\end{equation}
with an absolute constant \(C\).  In the present application take \(X=2M\) and
\(\ell=2a\).  Since \(a<2M\), we have \(\ell=2a\le4M=2X\), so the range condition in
\eqref{eq:additive-divisor-bound} is satisfied.  The short tail \(M<a<k\le2M\) is the
initial interval \(1\le h\le2M-a\), and enlarging it to \(h\le2M\) only increases the
nonnegative sum.

For completeness we recall the proof of \eqref{eq:additive-divisor-bound}.  Divisor
pairing gives
\[
        \tauD(n)\le 2\sum_{\substack{d\mid n\\d\le\sqrt n}}1.
\]
The left side of \eqref{eq:additive-divisor-bound} is bounded by a constant times
\[
        \sum_{d\le \sqrt X}\sum_{e\le\sqrt{X+\ell}}
        \#\{h\le X:h\equiv0\pmod d,
        \ h\equiv-\ell\pmod e\}.
\]
The congruences are soluble only when \(g=(d,e)\) divides \(\ell\), and then the number
of \(h\) is at most \(X/[d,e]+1\).  Writing \(d=gr\), \(e=gs\) and bounding the coprime
condition away, the \(X/[d,e]\) contribution is
\[
        \ll X\sum_{g\mid\ell}\frac1g
        \left(\sum_{r\le\sqrt X}\frac1r\right)
        \left(\sum_{s\le\sqrt{X+\ell}}\frac1s\right)
        \ll X(\log(2X))^2\sigma_{-1}(\ell).
\]
The \(+1\) contribution is \(O(X)\), because \(d\le\sqrt X\), \(e\le\sqrt{X+\ell}\),
and \(\ell\le2X\).  Finally
\[
        \sigma_{-1}(\ell)=\prod_{p^\nu\parallel\ell}\left(1+\frac1p+\cdots+\frac1{p^\nu}\right)
        \ll\prod_{p\mid\ell}\left(1+\frac1p\right)^C.
\]
With \(\ell=2a\), the last product is absorbed by \(\mathfrak D_{C_1'}(a)\).
\end{proof}

We next isolate the Euler-product mean value needed for the two-logarithm refinement.  The cited real-character value-distribution literature motivates this input, but the precise weighted statement below is uniform both in the conductor range and in the truncation parameter \(M\).  We do not derive that full long-truncation form from the cited results here.  Therefore the full uniform statement is kept as an explicit unproved input.

\begin{hypothesis}[Quadratic Euler-product mean-value input \(\HQE\)]
\label{hyp:quadratic-euler-input}
Let \(\lambda\ge0\) and \(C\ge0\) be fixed.  Uniformly for \(Y\ge2\) and \(M\ge2\),
\[
        \sum_{\substack{2\le s\le Y\\ s\textup{ squarefree}}}
        \mathfrak D_C(s)L_M(1,\chi_s)^\lambda \ll_{\lambda,C} Y,
\]
where \(\chi_s\) is the primitive real quadratic character attached to
\(\mathbb Q(\sqrt s)\).  The exclusion of \(s=1\) is essential, since the principal
truncated product is \(\asymp\log M\).
\end{hypothesis}

\begin{remark}[Status of \(\HQE\)]\label{rem:HQE-status}
The input \(\HQE\) is motivated by the Granville--Soundararajan moment and short
Euler-product theory, but it is not a formal consequence of the cited theorems without an
additional uniform long-truncation argument.  The proof of the present paper uses \(\HQE\)
only to pass from the elementary three-logarithm first moment to the two-logarithm first
moment.  All fixed-moment and tail estimates use the proved shifted-square theorem
\(\HST\) and do not require \(\HQE\).  A version of this manuscript that removes \(\HQE\)
should add a sourced theorem controlling the long partial Euler-product tail uniformly in
both conductor and truncation.
\end{remark}

\begin{lemma}[Harmonic mean of quadratic Euler products]\label{lem:quadratic-euler-mean}
Assume \Cref{hyp:quadratic-euler-input}.  Uniformly for $2\le M\le X$ and every fixed $C\ge0$,
\[
        \sum_{\substack{1\le j\le X\\ j\textup{ not a square}}}
        \frac{\mathfrak D_C(j)L_M(1,\chi_j)}j
        \ll_C \log X .
\]
\end{lemma}

\begin{proof}
Write every nonsquare $j$ uniquely as
\[
        j=s u^2,
        \qquad s>1\text{ squarefree}.
\]
By definition $\chi_j=\chi_s$, up to the harmless convention at the prime $2$, and
\[
        \mathfrak D_C(su^2)\le \mathfrak D_C(s)\mathfrak D_C(u).
\]
Hence
\[
\begin{aligned}
        \sum_{\substack{1\le j\le X\\j\textup{ not a square}}}
        \frac{\mathfrak D_C(j)L_M(1,\chi_j)}j
        &\le
        \sum_{u\le X^{1/2}}\frac{\mathfrak D_C(u)}{u^2}
        \sum_{\substack{s\le X/u^2\\s>1\textup{ squarefree}}}
        \frac{\mathfrak D_C(s)L_M(1,\chi_s)}s .
\end{aligned}
\]
By \Cref{hyp:quadratic-euler-input} with $\lambda=1$, the partial sums
\[
        A(Y)=\sum_{\substack{2\le s\le Y\\s\textup{ squarefree}}}
        \mathfrak D_C(s)L_M(1,\chi_s)
\]
satisfy $A(Y)\ll_C Y$, uniformly in $M$.  Partial summation gives
\[
        \sum_{\substack{s\le Y\\s\textup{ squarefree}}}
        \frac{\mathfrak D_C(s)L_M(1,\chi_s)}s\ll_C\log(2Y).
\]
Therefore the preceding display is
\[
        \ll_C \sum_{u\le X^{1/2}}\frac{\mathfrak D_C(u)}{u^2}
        \log\left(2X/u^2\right)
        \ll_C \log X,
\]
since $\sum_{u\ge1}\mathfrak D_C(u)/u^2$ converges.
\end{proof}

\begin{lemma}[Averaged shifted-square first moment]\label{lem:henriot-first-shifted}
Assume \Cref{hyp:quadratic-euler-input}.  As \(K\to\infty\),
\[
        \sum_{\substack{1\le j\le2K\\ j\textup{ not a square}}}
        \frac1j
        \sum_{\substack{k\le K\\ k^2>j}}\tauD(k^2-j)
        \ll K(\log K)^2.
\]
For the square shifts one also has
\[
        \sum_{a\le(2K)^{1/2}}\frac1{a^2}
        \sum_{\substack{k\le K\\ k>a}}\tauD(k^2-a^2)
        \ll K(\log K)^2.
\]
\end{lemma}

\begin{proof}
First take nonsquare \(j\), and decompose the \(k\)-sum into dyadic intervals
\(M<k\le2M\).  The finitely many intervals with \(M<2\) contribute \(O(K\log K)\),
which is harmless.  For \(M\ge2\), \Cref{thm:shifted-square-estimate}\textup{(b)} gives
\[
        \sum_{\substack{M<k\le2M\\ k^2>j}}\tauD(k^2-j)
        \ll
        M(\log(2M))\mathfrak D_{C_1'}(j)L_{2M}(1,\chi_j).
\]
Since \(2M\le2K\), Lemma \ref{lem:quadratic-euler-mean}, with \(X=2K\), gives
\[
\begin{aligned}
        &\sum_{\substack{1\le j\le2K\\j\text{ not a square}}}
        \frac1j
        \sum_{\substack{M<k\le2M\\k^2>j}}\tauD(k^2-j)  \\
        &\qquad\ll
        M\log(2M)
        \sum_{\substack{1\le j\le2K\\j\text{ not a square}}}
        \frac{\mathfrak D_{C_1'}(j)L_{2M}(1,\chi_j)}j
        \ll M\log(2M)\log K .
\end{aligned}
\]
Summing over dyadic \(M\le K\) gives the nonsquare contribution
\(O(K(\log K)^2)\).

For \(j=a^2\), \Cref{prop:square-shift-estimates}, summed over dyadic intervals,
gives
\[
        \sum_{\substack{k\le K\\ k>a}}\tauD(k^2-a^2)
        \ll K(\log K)^2\mathfrak D_{C_1'}(a).
\]
Therefore
\[
        \sum_{a\le(2K)^{1/2}}\frac1{a^2}
        \sum_{\substack{k\le K\\ k>a}}\tauD(k^2-a^2)
        \ll K(\log K)^2
        \sum_{a=1}^{\infty}\frac{\mathfrak D_{C_1'}(a)}{a^2}
        \ll K(\log K)^2,
\]
because \(\sum_a\mathfrak D_{C_1'}(a)/a^2\) is an absolutely convergent Euler product.
\end{proof}

\begin{theorem}[Conditional refinement under the unproved uniform hypothesis \(\HQE\)]\label{thm:E-average}
Assume \Cref{hyp:quadratic-euler-input}.  As \(K\to\infty\),
\[
        \sum_{2\le k\le K}|E_k|\ll K(\log K)^2.
\]
More strongly,
\[
        \sum_{2\le k\le K}a_k\ll K(\log K)^2.
\]
This theorem is not an unconditional result of the paper.  Its only new input beyond the
shifted-square estimate \(\HST\) is the hypothesis \(\HQE\).
The elementary root-counting bound is \Cref{thm:E-average-elementary}.  The saved
logarithm in the present theorem comes from combining \(\HST\) with \(\HQE\): one first applies the \(B=1\) part of
\Cref{thm:shifted-square-estimate} and then averages the remaining good-prime quadratic-character
factors over the deficits \(j\) using \Cref{hyp:quadratic-euler-input}.
\end{theorem}

\begin{proof}
For every active deficit \(j\) at level \(k\), one has \(j\le\tauD(k^2-j)\).  Thus
\[
        a_k\le \sum_{j=1}^{2k-2}\frac{\tauD(k^2-j)}j.
\]
Summing over \(k\le K\) and then over \(j\) gives
\[
        \sum_{2\le k\le K}a_k
        \le \sum_{j\le 2K}\frac1j
        \sum_{\substack{k\le K\\ k^2>j}}\tauD(k^2-j).
\]
The nonsquare part of the last sum is \(O(K(\log K)^2)\) by
\Cref{lem:henriot-first-shifted}.  The square part is
\[
        \sum_{a\le(2K)^{1/2}}\frac1{a^2}
        \sum_{\substack{k\le K\\ k>a}}\tauD(k^2-a^2),
\]
which is also \(O(K(\log K)^2)\) by the same lemma.  Hence
\(\sum_{2\le k\le K}a_k\ll K(\log K)^2\), and \(|E_k|\le a_k\) gives the
exit-set estimate.
\end{proof}

\begin{remark}[First-moment scale]\label{rem:first-moment-scale}
The exact identity in \Cref{prop:active-reindexing} rewrites the active-deficit first moment as a count of integers whose divisor count exceeds their distance to the next square.  This makes the conjectural scale \(K\log K\) natural: in the interval just below \(k^2\), the distance-to-square variable runs through \(1\le j\le 2k-2\), and if \(\tauD(k^2-j)\) behaved like a locally independent divisor count of average size about \(\log k^2\), then only about logarithmically many deficits should be active at a typical level.  A first moment of order \(K\log K\) is therefore compatible with the global average order of \(\tauD\).  The present paper proves the unconditional bounds
\[
        K\ll \sum_{2\le k\le K}a_k
        \quad\text{and}\quad
        \sum_{2\le k\le K}a_k\ll K(\log K)^3,
\]
and, assuming \(\HST+\HQE\), the refinement
\[
        \sum_{2\le k\le K}a_k\ll K(\log K)^2.
\]
The remaining extra logarithm in the two-logarithm bound is not merely a bookkeeping loss in the divisor average.  The active condition is the threshold event
\[
        \tauD(k^2-j)\ge j,
\]
so one needs information about the distribution of divisor values on the moving family of shifted squares \(k^2-j\), with the sharp weight imposed by distance to the next square.  Upper-bound sieve estimates can replace this indicator by \(\tauD(k^2-j)/j\); after summing over \(j\), that replacement introduces a harmonic logarithm which current upper-bound technology does not cancel.  Removing it would require an asymptotic or cancellation theorem for these conditioned indicators, strong enough that the accumulated error over \(1\le j\le 2K\) is \(o(K\log K)\).  The theorem \(\HST\), together with the input \(\HQE\), controls the expected local factors and saves one logarithm, but they remain upper-bound tools rather than such a distribution law.
\end{remark}

\begin{remark}[The elementary joint-root route]\label{rem:joint-root-route}
The most direct elementary approach to the second moment starts from
\[
        1_{\tauD(k^2-i)\ge i}\le \frac{\tauD(k^2-i)}{i}
\]
and hence
\[
        a_k^2\le
        \sum_{1\le i,j\le 2k-2}
        \frac{\tauD(k^2-i)\tauD(k^2-j)}{ij}.
\]
One is then led to count, for many pairs of moduli \(d,e\), the simultaneous congruences
\[
        k^2\equiv i\pmod d,\qquad k^2\equiv j\pmod e.
\]
The next proposition carries out this approach for truncated divisor sums.  It is included to isolate exactly what the root-counting method proves without invoking Nair--Tenenbaum type input.
\end{remark}

For \(Y\ge1\), put
\[
        \tauD_{\le Y}(n)=\sum_{\substack{d\mid n\\ d\le Y}}1
        \qquad(n\ge1),
\]
and extend it by \(\tauD_{\le Y}(n)=0\) for \(n\le0\).

\begin{proposition}[Elementary truncated joint-root second moment]\label{prop:truncated-joint-root}
There is an absolute constant \(C\) such that, uniformly for \(K\ge3\) and
\(1\le Y\le K^{1/2}\),
\[
        \sum_{2\le k\le K}
        \left(
        \sum_{1\le j\le 2k-2}\frac{\tauD_{\le Y}(k^2-j)}{j}
        \right)^2
        \ll K(\log K)^C .
\]
\end{proposition}

\begin{proof}
We enlarge the inner range to \(1\le j\le2K\).  Terms with \(j\ge k^2\) vanish by the zero-extension, and the remaining newly added positive terms only increase the quantity being bounded.  For
\(d,e\le Y\), write \([d,e]\) for the least common multiple, and define
\[
        N_{i,j}(d,e)=\#\{k\le K:k^2\equiv i\pmod d,
        \ k^2\equiv j\pmod e\}.
\]
Since \([d,e]\le Y^2\le K\), counting residue classes modulo \([d,e]\) gives
\[
        N_{i,j}(d,e)
        \le 2K\frac{\rho_i(d)\rho_j(e)}{[d,e]}.
\]
By \Cref{lem:root-bound},
\[
        \rho_i(d)\rho_j(e)
        \ll 2^{\omega(d)+\omega(e)}(i,d)^{1/2}(j,e)^{1/2}.
\]
Therefore the required second moment is at most
\[
\begin{aligned}
        &K\sum_{d,e\le Y}\frac{2^{\omega(d)+\omega(e)}}{[d,e]}
        \left(\sum_{i\le2K}\frac{(i,d)^{1/2}}{i}\right)
        \left(\sum_{j\le2K}\frac{(j,e)^{1/2}}{j}\right)  \\
        &\qquad\ll
        K(\log K)^2
        \sum_{d,e\le Y}\frac{F(d)F(e)}{[d,e]},
\end{aligned}
\]
where
\[
        F(n)=2^{\omega(n)}\sum_{a\mid n}a^{-1/2}.
\]
Here we used
\[
        \sum_{i\le2K}\frac{(i,d)^{1/2}}{i}
        \le \sum_{a\mid d}a^{1/2}\sum_{\substack{i\le2K\\a\mid i}}\frac{1}{i}
        \ll (\log K)\sum_{a\mid d}a^{-1/2}.
\]
It remains to note that the double lcm sum is polylogarithmic.  Since
\(F(n)\le \tauD(n)^2\), writing \(d=ha\), \(e=hb\), \((a,b)=1\), gives
\[
\begin{aligned}
        \sum_{d,e\le Y}\frac{F(d)F(e)}{[d,e]}
        &\le
        \sum_{h\le Y}\frac{\tauD(h)^4}{h}
        \left(\sum_{a\le Y/h}\frac{\tauD(a)^2}{a}\right)^2  \\
        &\ll (\log Y)^C,
\end{aligned}
\]
by the standard elementary divisor-moment estimates
\(\sum_{n\le x}\tauD(n)^A/n\ll_A(\log x)^{C_A}\).  This proves the claim.
\end{proof}

\begin{remark}[Why this elementary route is not enough]\label{rem:joint-root-barrier}
The preceding proposition is the clean part of the proposed elementary second-moment proof.  To use it for the full divisor function one would need to take divisor moduli as large as \(K\).  Then \([d,e]\) can be as large as \(K^2\), and the elementary residue-class estimate becomes
\[
        N_{i,j}(d,e)\le \left(\frac{K}{[d,e]}+1\right)
        \rho_{i,j}(d,e).
\]
The additional \(+1\) term no longer satisfies \(+1\ll K/[d,e]\), and after insertion of the weights \((ij)^{-1}\) it leads to a large-lcm remainder not controlled by \Cref{lem:root-bound} alone.  Thus this elementary route does not by itself prove
\(\sum_{k\le K}a_k^2\ll K(\log K)^C\); the untruncated fixed-moment estimates below use the shifted-square theorem \(\HST\).
\end{remark}

\begin{corollary}[Unconditional density-one polylogarithmic normal order]\label{cor:E-normal-order-unconditional}
For every \(\varepsilon>0\),
\[
        \#\{k\le K:|E_k|>(\log k)^{3+\varepsilon}\}=o(K).
\]
Equivalently, \(|E_k|\le(\log k)^{3+\varepsilon}\) for a natural-density-one set of
levels \(k\).
\end{corollary}

\begin{proof}
The contribution of \(k\le K^{1/2}\) is \(o(K)\).  For \(K^{1/2}<k\le K\), we have
\(\log k\ge \frac12\log K\).  The assertion follows from
\Cref{thm:E-average-elementary}:
\[
\#\{K^{1/2}<k\le K:|E_k|>(\log k)^{3+\varepsilon}\}
\ll_\varepsilon \frac{K(\log K)^3}{(\log K)^{3+\varepsilon}}
=o(K).
\]
\end{proof}

\begin{corollary}[Conditional \(\HQE\) density-one polylogarithmic normal order]\label{cor:E-normal-order}
Assume \Cref{hyp:quadratic-euler-input}.  For every \(\varepsilon>0\),
\[
        \#\{k\le K:|E_k|>(\log k)^{2+\varepsilon}\}=o(K).
\]
Equivalently, \(|E_k|\le(\log k)^{2+\varepsilon}\) for a natural-density-one set of
levels \(k\).
\end{corollary}

\begin{proof}
The proof is the same dyadic-density argument, using the refined theorem
\Cref{thm:E-average} in place of \Cref{thm:E-average-elementary}.
\end{proof}

The first-moment estimates above are complemented below by fixed-moment and truncated first-moment bounds.  Those bounds improve the control of large values, but they do not prove a linear first moment.

The next lemma records the full-interval form needed for the fixed-moment estimates and repeats the local bad-prime calculation so that the dependence on the shift \(j\) is visible inside the proof.

Here is the heuristic behind the discriminant factor.  Expanding a divisor sum over \(Q_j(k)=k^2-j\) asks, for each modulus \(d\), how often \(Q_j(k)\equiv0\pmod d\).  By the Chinese remainder theorem this is controlled prime by prime by the number of roots of
\[
        x^2\equiv j\pmod {p^\nu}.
\]
If \(p\nmid 2j\), the two roots, when they exist, are simple and Hensel lifting is harmless; these are the good primes and they contribute only to the usual logarithmic Euler products.  If \(p\mid2j\), the derivative \(2x\) may vanish at a root and the roots can coalesce or proliferate at higher prime powers.  These are exactly the primes at which the congruence classes bottleneck the divisor sum.  Since the discriminant of \(X^2-j\) is \(4j\), the bad-prime weight in \Cref{thm:shifted-square-estimate} depends only on the primes dividing \(2j\), which is why the factor \(\mathfrak D_A(j)\) appears below.

\begin{lemma}[Discriminant-explicit shifted-square divisor moment]\label{lem:shifted-square-moment}
Let \(A\ge1\) be fixed.  There are constants \(C_0,C_1\), depending only on \(A\), such that, uniformly for \(K\ge2\) and \(1\le j\le2K\),
\[
        \sum_{\substack{k\le K\\ k^2>j}} \tauD(k^2-j)^A
        \ll_A K(\log K)^{C_1}\mathfrak D_A(j),
        \qquad
        \mathfrak D_A(j):=\prod_{p\mid 2j}\left(1+\frac1p\right)^{C_0}.
\]
Consequently
\[
        \sum_{j=1}^{\infty}\frac{\mathfrak D_A(j)}{j^2}<\infty .
\]
\end{lemma}

\begin{proof}
Choose the exponent in the definition of \(\mathfrak D_A\) large enough to dominate the bad-prime exponent \(C_A'\) in \Cref{thm:shifted-square-estimate}.  If \(j\) is not a square,
decompose \(1\le k\le K\) into dyadic intervals.  On every contributing interval
\(M<k\le2M\) one has \(j<4M^2\), so \Cref{thm:shifted-square-estimate}\textup{(a)} applies and gives
\[
        \sum_{\substack{M<k\le2M\\ k^2>j}}
        \tauD(k^2-j)^A
        \ll_A M(\log(2M))^{C_A}\mathfrak D_A(j).
\]
The initial interval \(M<2\) is absorbed into the constant, and summing over dyadic
\(M\le K\) yields
\[
        \sum_{\substack{k\le K\\ k^2>j}}\tauD(k^2-j)^A
        \ll_A K(\log K)^{C_1}\mathfrak D_A(j)
\]
for a possibly larger exponent \(C_1\).

If \(j=a^2\), \Cref{prop:square-shift-estimates}, summed over dyadic intervals,
gives
\[
        \sum_{\substack{k\le K\\ k>a}}\tauD(k^2-a^2)^A
        \ll_A K(\log K)^{C_1}\mathfrak D_A(a).
\]
Since every prime divisor of \(a\) also divides \(a^2=j\), we have
\(\mathfrak D_A(a)\le\mathfrak D_A(j)\), after increasing the fixed exponent if necessary.
This proves the uniform shifted-square estimate in both the irreducible and square-shift
cases.

Finally, the prime \(2\) contributes at most a fixed factor, and
\[
        \sum_{j=1}^{\infty}\frac{\mathfrak D_A(j)}{j^2}
        \le C
        \prod_p\left(1+
          \sum_{\nu\ge1}\frac{(1+1/p)^{C_0}}{p^{2\nu}}\right)<\infty.
\]
Each local factor is \(1+O(p^{-2})\), so the Euler product converges.
\end{proof}

\begin{theorem}[Second moment of the exit sets]\label{thm:E-second-moment}
There is an absolute constant \(C\ge0\) such that
\[
        \sum_{2\le k\le K}|E_k|^2\le \sum_{2\le k\le K}a_k^2
        \ll K(\log K)^C .
\]
\end{theorem}

\begin{proof}
If \(j\) is active at level \(k\), then \(j\le\tauD(k^2-j)\).  Hence
\[
        1_{\tauD(k^2-j)\ge j}\le \frac{\tauD(k^2-j)^2}{j^2},
\]
and therefore
\[
        a_k\le \sum_{1\le j\le 2k-2}\frac{\tauD(k^2-j)^2}{j^2}.
\]
By Cauchy's inequality and \(\sum_{j\ge1}j^{-2}<\infty\),
\[
        a_k^2
        \ll \sum_{1\le j\le 2k-2}\frac{\tauD(k^2-j)^4}{j^2}.
\]
Summing over \(k\le K\) and using Lemma \ref{lem:shifted-square-moment} with \(A=4\),
\[
\begin{aligned}
        \sum_{2\le k\le K}a_k^2
        &\ll \sum_{j\le 2K}\frac1{j^2}
              \sum_{\substack{k\le K\\ k^2>j}}\tauD(k^2-j)^4  \\
        &\ll K(\log K)^C
              \sum_{j\le2K}\frac{\mathfrak D_4(j)}{j^2}
        \ll K(\log K)^C.
\end{aligned}
\]
Since \(|E_k|\le a_k\), the same bound holds for the exit sets.
\end{proof}

\begin{theorem}[Fixed moments of the exit sets]\label{thm:E-fixed-moments}
For every fixed integer \(m\ge2\), there is a constant \(C_m\ge0\) such that
\[
        \sum_{2\le k\le K}|E_k|^m
        \le \sum_{2\le k\le K}a_k^m
        \ll_m K(\log K)^{C_m}.
\]
\end{theorem}

\begin{proof}
As in the second-moment proof,
\[
        a_k\le \sum_{1\le j\le2k-2}\frac{\tauD(k^2-j)^2}{j^2}.
\]
Put \(w_j=j^{-2}\).  By H\"older's inequality in the form
\[
        \left(\sum_j w_j y_j\right)^m
        \le \left(\sum_j w_j\right)^{m-1}\sum_j w_j y_j^m,
\]
with \(y_j=\tauD(k^2-j)^2\), and using \(\sum_j j^{-2}<\infty\), we obtain
\[
        a_k^m\ll_m
        \sum_{1\le j\le2k-2}\frac{\tauD(k^2-j)^{2m}}{j^2}.
\]
Summing over \(k\le K\), interchanging the order of summation, and applying Lemma \ref{lem:shifted-square-moment} with \(A=2m\),
\[
\begin{aligned}
        \sum_{2\le k\le K}a_k^m
        &\ll_m \sum_{j\le2K}\frac1{j^2}
        \sum_{\substack{k\le K\\ k^2>j}}\tauD(k^2-j)^{2m} \\
        &\ll_m K(\log K)^{C_m}
        \sum_{j\le2K}\frac{\mathfrak D_{2m}(j)}{j^2}
        \ll_m K(\log K)^{C_m}.
\end{aligned}
\]
The final series is bounded uniformly in \(K\) by Lemma \ref{lem:shifted-square-moment}.  Since \(|E_k|\le a_k\), the same estimate holds for \(|E_k|^m\).
\end{proof}

\begin{theorem}[Truncated first moment and large-value contribution]\label{thm:E-truncated-first}
Let \(m\ge2\) be fixed, and let \(C_m\) be as in Theorem \ref{thm:E-fixed-moments}.  Uniformly for \(V\ge1\),
\[
        \sum_{\substack{2\le k\le K\\ |E_k|>V}} |E_k|
        \ll_m \frac{K(\log K)^{C_m}}{V^{m-1}},
\]
and hence
\[
        \sum_{2\le k\le K}|E_k|
        \ll_m K V+\frac{K(\log K)^{C_m}}{V^{m-1}}.
\]
The same estimates hold with \(|E_k|\) replaced by the active-deficit count \(a_k\), and also with \(|E_k|\) replaced by the one-step width \(|\AnnT_k(E_k)|\).  In particular,
\[
        \sum_{2\le k\le K}|E_k|
        \ll_m K(\log K)^{C_m/m}.
\]
Moreover, for every fixed \(\Lambda>C_m/(m-1)\), the total contribution of levels with
\(|E_k|>(\log K)^\Lambda\) is \(o(K)\).
\end{theorem}

\begin{proof}
For the large-value contribution, use \(|E_k|\le |E_k|^m/V^{m-1}\) on the set \(|E_k|>V\), and then apply Theorem \ref{thm:E-fixed-moments}.  The full first-moment bound follows by splitting the levels according to whether \(|E_k|\le V\) or \(|E_k|>V\).  The active-deficit version is identical, using the moment estimate for \(a_k\) in Theorem \ref{thm:E-fixed-moments}; the one-step width version follows from \(|\AnnT_k(E_k)|\le |E_k|\).  Taking \(V=(\log K)^{C_m/m}\) gives the displayed interpolation bound.  Finally, if \(\Lambda>C_m/(m-1)\), then
\[
        \sum_{\substack{2\le k\le K\\ |E_k|>(\log K)^\Lambda}} |E_k|
        \ll_m K(\log K)^{C_m-\Lambda(m-1)}=o(K).
\]
\end{proof}

\begin{remark}[Limits of the average theory]\label{rem:first-moment-barrier}
Theorem \ref{thm:E-truncated-first} packages the fixed-moment information into first-moment envelopes and shows that sufficiently large exit sets have negligible total contribution.  The constants in \Cref{thm:shifted-square-estimate} are not optimized here.  The elementary first-moment exponent available by root counting is \(3\); assuming \(\HST+\HQE\) it is \(2\).  In particular, these results do not prove \(\sum_{k\le K}|E_k|=O(K)\).  That linear first-moment bound remains open.  On the other hand, Corollary \ref{cor:excess-not-o} shows that the stronger density-minimal assertion \(\sum_{k\le K}(|E_k|-2)=o(K)\) is false, so any eventual linear first-moment theorem must have limiting average strictly larger than \(2\).
\end{remark}

\begin{corollary}[Fixed-moment tail bounds]\label{cor:E-fixed-tail}
For every fixed integer \(m\ge2\), uniformly for \(V\ge1\),
\[
        \#\{2\le k\le K: |E_k|>V\}
        \ll_m \frac{K(\log K)^{C_m}}{V^m}.
\]
The same estimate holds with \(|E_k|\) replaced by the one-step width \(|\AnnT_k(E_k)|\).
\end{corollary}

\begin{proof}
This is Markov's inequality applied to Theorem \ref{thm:E-fixed-moments}; the final assertion follows from \(|\AnnT_k(E_k)|\le |E_k|\).
\end{proof}

\begin{corollary}[Quantitative square-tail bound]\label{cor:E-square-tail}
With the constant \(C\) from Theorem \ref{thm:E-second-moment}, uniformly for \(V\ge1\),
\[
        \#\{2\le k\le K: |E_k|>V\}
        \ll \frac{K(\log K)^C}{V^2}.
\]
In particular, for every \(\varepsilon>0\), \(|E_k|\le (\log k)^{C/2+\varepsilon}\) for a natural-density-one set of levels.
\end{corollary}

\begin{proof}
This is Chebyshev's inequality applied to Theorem \ref{thm:E-second-moment}, followed by the usual dyadic decomposition for the density-one assertion.
\end{proof}

\begin{corollary}[Average and moments for one-step widths]\label{cor:average-confluence-width}
With \(U_k=\AnnT_k(E_k)\), unconditionally
\[
        \sum_{2\le k\le K}|U_k|\ll K(\log K)^3.
\]
There is a constant \(C\) such that
\[
        \sum_{2\le k\le K}|U_k|^2\ll K(\log K)^C.
\]
Assuming \(\HST+\HQE\), the first bound improves to
\[
        \sum_{2\le k\le K}|U_k|\ll K(\log K)^2.
\]
Consequently the same square-tail, fixed-moment tail, and truncated
first-moment estimates hold for \(|U_k|\), with the same constants as the corresponding estimates for \(|E_k|\).
\end{corollary}

\begin{proof}
These estimates follow immediately from
\(|U_k|=|\AnnT_k(E_k)|\le |E_k|\), from the elementary average theorem
\Cref{thm:E-average-elementary}, from the refined theorem \Cref{thm:E-average}, and from the second-moment theorem
\Cref{thm:E-second-moment}.  The fixed-moment tail estimates follow in the same way
from Corollary \ref{cor:E-fixed-tail}.
\end{proof}

\subsection*{The parity obstruction to two-branch collapse}

The moment estimates above control large values of the one-step width, but they
cannot by themselves prove two-branch collapse.  The transfer
parity law gives an exact decomposition into two same-parity families, and the
remaining issue is genuinely dynamical: within each family, different offsets
must be shown to have the same first-entry image.

For a finite set of nonnegative offsets \(A\), write
\[
        A^{\EvPlus}=A\cap 2\Z\cap\{1,2,3,\ldots\},\qquad
        A^{\OddZero}=\bigl(A\cap(2\Z+1)\bigr)\cup(A\cap\{0\}).
\]
These are source classes, named by the parity mechanism in \Cref{prop:parity}:
the transfer images of \(E_k^{\EvPlus}\) are odd offsets in the next annulus,
whereas the transfer images of \(E_k^{\OddZero}\) are even offsets.

\begin{proposition}[Parity decomposition of the one-step frontier]\label{prop:exact-two-branch-criterion}
For every \(k\ge2\), put \(U_k=\AnnT_k(E_k)\).  Then
\[
        U_k
        =\AnnT_k(E_k^{\EvPlus})
         \sqcup \AnnT_k(E_k^{\OddZero}),
\]
where the first set consists only of odd offsets and the second only of even offsets.  In particular,
\[
        |U_k|
        =\left|\AnnT_k(E_k^{\EvPlus})\right|
         +\left|\AnnT_k(E_k^{\OddZero})\right|.
\]
Consequently, if both same-parity families have at most one image, then \(|U_k|\le2\).  Conversely, any failure of this same-parity condition either forces \(|U_k|>2\), or else \(U_k\) is a two-point set whose two elements have the same parity.
\end{proposition}

\begin{proof}
By the parity law, if \(r>0\), then
\[
        \AnnT_k(r)\equiv r+1\pmod2,
\]
while \(\AnnT_k(0)\) is even.  Thus positive even offsets map to odd offsets, and odd offsets together with the possible offset \(0\) map to even offsets.  The two displayed image sets are therefore disjoint, and the cardinality formula follows.

If each image set has size at most one, their disjoint union has size at most two.  Conversely, suppose one of the two image sets has at least two elements.  If the other image set is nonempty, then \(|U_k|>2\).  If the other image set is empty and \(|U_k|\le2\), then \(U_k\) consists of exactly two elements, both lying in the same parity class.  This proves the final assertion.
\end{proof}

\begin{definition}[Same-parity annular coalescence]\label{def:same-parity-annular-coalescence}
A level \(k\ge2\) is called \emph{same-parity coalescent} if each of the two sets
\(E_k^{\EvPlus}\) and \(E_k^{\OddZero}\) has at most one image under
\(\AnnT_k\).
\end{definition}

\begin{definition}[One-step two-branch collapse]\label{def:two-branch-collapse}
We say that level \(k\ge2\) has \emph{one-step two-branch collapse} if
\[
        |\AnnT_k(E_k)|\le2.
\]
We say that one-step two-branch collapse holds on a density-one set of levels if
\[
        \#\{2\le k\le K:|\AnnT_k(E_k)|>2\}=o(K).
\]
\end{definition}

\begin{corollary}[Two-branch collapse versus same-parity coalescence]\label{cor:two-branch-equivalence}
Same-parity coalescence at level \(k\) implies \(|\AnnT_k(E_k)|\le2\).  Moreover,
\[
\begin{aligned}
&\#\{2\le k\le K:\text{ level }k\text{ is not same-parity coalescent}\}\\
&\qquad\le
\#\{2\le k\le K:|\AnnT_k(E_k)|>2\}
 +\#\{2\le k\le K:\AnnT_k(E_k)\text{ is a two-point same-parity set}\}.
\end{aligned}
\]
Hence density-one same-parity coalescence implies density-one two-branch collapse.  Conversely, density-one two-branch collapse implies density-one same-parity coalescence provided the two-point same-parity exceptional levels have density zero.
\end{corollary}

\begin{proof}
The first assertion follows immediately from \Cref{prop:exact-two-branch-criterion}.  For the counting inequality, a non-same-parity-coalescent level either has \(|U_k|>2\), or else the final alternative in \Cref{prop:exact-two-branch-criterion} holds.  The density statements are the corresponding asymptotic forms of this inequality.
\end{proof}

\begin{corollary}[Large-excess bounds for one-step widths]\label{cor:U-excess-moments}
For every fixed integer \(m\ge2\) there is \(C_m\ge0\) such that
\[
        \sum_{2\le k\le K}\bigl(|U_k|-2\bigr)_+^m
        \ll_m K(\log K)^{C_m}.
\]
Consequently, uniformly for \(V\ge1\),
\[
        \#\{2\le k\le K: |U_k|>2+V\}
        \ll_m \frac{K(\log K)^{C_m}}{V^m}.
\]
\end{corollary}

\begin{proof}
Since \((|U_k|-2)_+\le |U_k|\le |E_k|\), the moment bound follows from
\Cref{thm:E-fixed-moments}.  The tail estimate is Markov's inequality.
\end{proof}

\begin{remark}[What remains open in the two-branch problem]\label{rem:two-branch-open}
\Cref{prop:exact-two-branch-criterion} shows why two-branch collapse is more rigid than a smallness statement for \(E_k\): one must prove that positive even exit offsets usually share one first-entry image, and that odd exit offsets, together with the possible offset \(0\), usually share one first-entry image.  The fixed-moment bounds in \Cref{cor:U-excess-moments} rule out frequent large excess, but they do not reach the fixed threshold \(|U_k|>2\), nor do they rule out a two-point same-parity exceptional image.  A proof of density-one two-branch collapse therefore requires genuinely dynamical information about same-parity annular coalescence, not only divisor-moment estimates for the static frontier \(E_k\).
\end{remark}

\begin{remark}[Confluence interpretation]
The one-step set \(U_k=\AnnT_k(E_k)\) is the case \(s=1\) of Definition \ref{def:confluence-widths}.  Proposition \ref{prop:dynamic-frontier} gives \(R(k^2-1)\le |U_k|\), and Corollary \ref{cor:dynamic-confluence} shows that \(\liminf |U_k|=1\) would imply connectedness.  The moment bounds, including all fixed moments, supply typical-size and tail information for this dynamical bottleneck, but they do not by themselves imply a density-one two-branch theorem.
\end{remark}

\section{Lower-runner races}\label{sec:races}

Consider two forward orbits.  At a race state, let their current values be \(p<q\).  Advance the lower value \(p\) to \(\Tmap(p)\), then reorder.  If \(\Tmap(p)=q\), the two orbits coalesce and the race terminates.  Otherwise, writing \(g=q-p\), the next gap is
\begin{equation}\label{eq:gap-update}
        |\tauD(p)-g|
\end{equation}
and the next lower value is
\begin{equation}\label{eq:lower-update}
        p+\min(g,\tauD(p)).
\end{equation}

\begin{lemma}\label{lem:race-detects}
The lower-runner race terminates if and only if the two forward orbits coalesce.
\end{lemma}

\begin{proof}
Let the two orbits be \(x_i=\Tmap^i(x_0)\) and \(y_j=\Tmap^j(y_0)\).  If the
race terminates, the current two values are equal, so the two orbits coalesce.

Conversely, suppose the two orbits coalesce, and let \(w\) be their least common value.
Choose the first indices \(I,J\) such that
\[
        x_I=y_J=w .
\]
Before index \(I\), the \(x\)-orbit is strictly below \(w\); before index \(J\), the
\(y\)-orbit is strictly below \(w\).  During the lower-runner race, as long as termination
has not occurred, the race advances exactly one of the two current orbit indices, and it
never advances an orbit value equal to \(w\) while the other current value is still below
\(w\), because \(w\) is then the upper value.  Hence the pair of indices remains in the
rectangle
\[
        0\le i\le I,\qquad 0\le j\le J
\]
until termination.  Since at each nonterminal step one of the indices increases, after at
most \(I+J\) steps both current values are \(w\), and the race terminates.
\end{proof}

\section{CRT blocks and large jumps}\label{sec:large-jumps}

We need a quantitative form of the standard Chinese-remainder construction of short blocks in which every integer has many divisors.  The primorial form below keeps the scale explicit.

\begin{definition}[Primorial block modulus]
Let \(p_1,p_2,\ldots\) denote the primes in increasing order and put
\[
        P_K=\prod_{u\le K}p_u .
\]
For integers \(L\ge1\), \(B\ge2\), set
\[
        a(B)=\left\lceil\frac{\log(B+1)}{\log 2}\right\rceil,
        \qquad K(L,B)=L a(B).
\]
\end{definition}

\begin{lemma}[High-divisor blocks, primorial form]\label{lem:crt-block}
For every pair of integers \(L\ge1\), \(B\ge2\), and every \(Z\ge1\), there exists an integer \(M\) with
\[
        Z\le M<Z+P_{K(L,B)}
\]
such that
\[
        \tauD(M+i)>B\qquad(1\le i\le L).
\]
Moreover, as \(K\to\infty\),
\begin{equation}\label{eq:Q-asymp}
        \log P_K=(1+o(1))K\log K,
\end{equation}
and hence, whenever \(L a(B)\to\infty\),
\begin{equation}\label{eq:crt-asymp-general}
        \log P_{K(L,B)}=
        (1+o(1))L a(B)\log(L a(B)).
\end{equation}
If in addition \(B\to\infty\), then
\begin{equation}\label{eq:crt-asymp}
        \log P_{K(L,B)}=
        \left(\frac{1}{\log 2}+o(1)\right)L\log B\log(L\log B).
\end{equation}
In particular, there is an absolute constant \(C>0\) such that
\begin{equation}\label{eq:crt-uniform}
        P_{K(L,B)}\le \exp\!\bigl(C L\log B\log(2L\log B)\bigr)
        \qquad(L\ge1,\ B\ge2).
\end{equation}
\end{lemma}

\begin{proof}
Let \(a=a(B)\), so that \(2^a>B\), and let \(K=La\).  Split the first \(K\) primes into \(L\) disjoint consecutive blocks of length \(a\), and let \(Q_i\) be the product of the primes in the \(i\)-th block.  Then the \(Q_i\) are pairwise coprime, squarefree, and each has exactly \(a\) prime factors, while \(\prod_i Q_i=P_K\).  By the Chinese remainder theorem there is a residue class \(M_0\pmod {P_K}\) satisfying
\[
        M_0+i\equiv0\pmod {Q_i}\qquad(1\le i\le L).
\]
Choose \(M\equiv M_0\pmod {P_K}\) with \(Z\le M<Z+P_K\).  Then \(Q_i\mid M+i\), so \(M+i\) has at least \(a\) distinct prime divisors and therefore
\[
        \tauD(M+i)\ge2^a>B.
\]
The prime number theorem gives \(\sum_{u\le K}\log p_u=(1+o(1))K\log K\), proving \eqref{eq:Q-asymp}.  Substituting \(K=L a(B)\) gives \eqref{eq:crt-asymp-general} when \(L a(B)\to\infty\).  If also \(B\to\infty\), then \(a(B)=(1/\log 2+o(1))\log B\), and \eqref{eq:crt-asymp} follows.  The uniform bound \eqref{eq:crt-uniform} follows, for example, from the standard estimate \(p_K\ll K\log(2K)\).
\end{proof}

\begin{theorem}[Large jumps on every orbit]\label{thm:orbit-jumps}
For every starting value \(n\in\N\) and every integer \(G\ge2\), some iterate \(x=\Tmap^t(n)\) satisfies
\begin{equation}\label{eq:orbit-exact-scale}
        x\le n+P_{(G+1)\left\lceil\frac{\log(G+1)}{\log 2}\right\rceil}+G+1
        \qquad\text{and}\qquad
        \tauD(x)>G.
\end{equation}
In particular, for some absolute constant \(C>0\),
\begin{equation}\label{eq:orbit-uniform-scale}
        x\le n+\exp\!\bigl(CG\log G\log(G\log G)\bigr),
\end{equation}
and more sharply
\begin{equation}\label{eq:orbit-asymp-scale}
        \log P_{(G+1)\left\lceil\frac{\log(G+1)}{\log 2}\right\rceil}
        =\left(\frac{1}{\log 2}+o(1)\right)G\log G\log(G\log G).
\end{equation}
Consequently, for every fixed starting value \(n\),
\begin{equation}\label{eq:orbit-growth}
        \max_{\substack{t\ge0\\\Tmap^t(n)\le n+Y}}\tauD(\Tmap^t(n))
        \ge (\log 2-o(1))\frac{\log Y}{(\log\log Y)^2}
\end{equation}
as \(Y\to\infty\), and in particular \(\limsup_{t\to\infty}\tauD(\Tmap^t(n))=\infty\).
\end{theorem}

\begin{proof}
Apply Lemma \ref{lem:crt-block} with \(L=G+1\), \(B=G\), and \(Z=n\).  We get a block
\[
        M+1,\ldots,M+G+1
\]
with \(M<n+P_{(G+1)\left\lceil\frac{\log(G+1)}{\log 2}\right\rceil}\), such that every integer in the block has more than \(G\) divisors.  Follow the orbit \(x_{t+1}=x_t+\tauD(x_t)\).  If \(\tauD(x_t)>G\) before the orbit passes the block, we are done.  Otherwise every step up to that point has length at most \(G\), so the strictly increasing orbit cannot jump over \(G+1\) consecutive integers.  It must land in the block, where \(\tauD>G\).  The uniform and asymptotic scales follow from Lemma \ref{lem:crt-block}.  To invert \eqref{eq:orbit-asymp-scale}, fix \(\eta>0\) and take \(G=\lfloor(\log 2-\eta)\log Y/(\log\log Y)^2\rfloor\); then the right-hand side of \eqref{eq:orbit-asymp-scale} is at most \((1+o(1))(1-\eta/\log 2)\log Y\), so the required iterate lies below \(n+Y\) for all sufficiently large \(Y\).  Letting \(\eta\downarrow0\) gives \eqref{eq:orbit-growth}.
\end{proof}

For a set \(A\subseteq\N_0\) of time indices, write
\[
        \underline d(A)=\liminf_{N\to\infty}\frac{|A\cap\{0,1,\ldots,N-1\}|}{N}
\]
for its lower asymptotic density.  In orbit statements this density is taken in
iterate time; in race statements it is taken in race-move time.

\begin{corollary}[Syndetic large jumps along every orbit]\label{cor:syndetic-orbit-jumps}
Fix \(G\ge2\), and set
\[
        Q_G=P_{(G+1)\left\lceil\frac{\log(G+1)}{\log 2}\right\rceil},
        \qquad L_G=Q_G+G+1.
\]
For every orbit \(x_t=\Tmap^t(n)\) and every index \(s\ge0\), there is an index \(t\) with
\[
        s\le t\le s+L_G
        \qquad\text{and}\qquad
        \tauD(x_t)>G.
\]
Consequently the set \(\{t\ge0:\tauD(\Tmap^t(n))>G\}\) has lower asymptotic density at least \((L_G+1)^{-1}\).
\end{corollary}

\begin{proof}
Apply Theorem \ref{thm:orbit-jumps} with starting value \(x_s\).  It gives an iterate \(x_t\) with \(\tauD(x_t)>G\) and \(x_t\le x_s+L_G\).  Since every iterate increases by at least one, \(t-s\le x_t-x_s\le L_G\).  Thus every block of \(L_G+1\) consecutive time indices contains such an event.
\end{proof}

\begin{lemma}[Unbounded lower values in a nonterminal race]\label{lem:race-lower-values-unbounded}
In any lower-runner race that does not terminate, the sequence of lower values is strictly increasing and unbounded.
\end{lemma}

\begin{proof}
At a race state with current values \(p<q\), the next lower value is
\(p+\min(q-p,\tauD(p))\), unless the race terminates.  Since both terms in the minimum are positive in a nonterminal move, the lower value strictly increases.  If the lower values were bounded, then this strictly increasing integer sequence would have only finitely many possible values and would eventually stop, impossible in a nonterminal race.  Equivalently, the current race values are always values of two strictly increasing unbounded forward orbits, so they cannot remain trapped below a fixed bound.
\end{proof}

\begin{theorem}[Quantitative lower-runner divergence]\label{thm:race-quant}
Suppose two \(\Tmap\)-orbits do not coalesce, and run the lower-runner race between them.  If \(p_0\) is the lower initial value, then for every integer \(G\ge2\) there is a race state with gap greater than \(G\) and lower value at most
\begin{equation}\label{eq:race-exact-scale}
        p_0+P_{(G+1)\left\lceil\frac{\log(2G+1)}{\log 2}\right\rceil}+2G+1.
\end{equation}
In particular, for some absolute constant \(C>0\), this lower value is at most
\begin{equation}\label{eq:race-uniform-scale}
        p_0+\exp\!\bigl(CG\log G\log(G\log G)\bigr),
\end{equation}
and the logarithm of the primorial part of this scale is
\begin{equation}\label{eq:race-asymp-scale}
        \log P_{(G+1)\left\lceil\frac{\log(2G+1)}{\log 2}\right\rceil}
        =\left(\frac{1}{\log 2}+o(1)\right)G\log G\log(G\log G).
\end{equation}
In particular, the lower-runner gap is unbounded.
\end{theorem}

\begin{proof}
Apply Lemma \ref{lem:crt-block} with \(L=G+1\), \(B=2G\), and \(Z=p_0\).  This produces a block \(M+1,\ldots,M+G+1\), with \(M<p_0+P_{(G+1)\left\lceil\frac{\log(2G+1)}{\log 2}\right\rceil}\), on which \(\tauD>2G\).

Assume for contradiction that every race state with lower value at most \(M+2G+1\) has gap at most \(G\).  By \eqref{eq:lower-update}, the lower-runner sequence then increases by at most \(G\) at each race step.  By \Cref{lem:race-lower-values-unbounded}, the lower values are strictly increasing and unbounded.  Therefore they cannot jump over the block of length \(G+1\).  At some race state the lower value is \(p=M+i\), \(1\le i\le G+1\), and the gap is \(g\le G\).  Since \(\tauD(p)>2G\), the next gap is \(\tauD(p)-g>G\).  The lower value at the next race state is the previous upper value, namely \(p+g\le M+2G+1\), contradicting the assumption.  The displayed bounds follow from Lemma \ref{lem:crt-block}.
\end{proof}

\begin{corollary}[Syndetic large gaps in every non-coalescing race]\label{cor:syndetic-race-gaps}
Fix \(G\ge2\), and set
\[
        Q'_G=P_{(G+1)\left\lceil\frac{\log(2G+1)}{\log 2}\right\rceil},
        \qquad L'_G=Q'_G+2G+1.
\]
In any non-coalescing lower-runner race, from every race state, the current state or some later state within at most \(L'_G\) race moves has gap greater than \(G\).  Consequently the set of race times at which the gap exceeds \(G\) has lower asymptotic density at least \((L'_G+1)^{-1}\).
\end{corollary}

\begin{proof}
If the current race state already has gap greater than \(G\), there is nothing to prove.  Otherwise restart Theorem \ref{thm:race-quant} at the current race state, taking its lower value as the new initial lower value.  The theorem gives a later race state with lower value increased by at most \(L'_G\) and gap greater than \(G\).  Since the lower value increases by at least one at each non-terminal race move, this later state occurs within at most \(L'_G\) race moves.  The lower-density conclusion follows as in Corollary \ref{cor:syndetic-orbit-jumps}.
\end{proof}

\begin{corollary}\label{cor:race-growth}
For every non-coalescing pair of orbits, if \(G_{\max}(Y)\) denotes the maximum race gap among states whose lower value is at most \(p_0+Y\), then
\[
        G_{\max}(Y)\ge (\log 2-o(1))\frac{\log Y}{(\log\log Y)^2}
        \qquad(Y\to\infty).
\]
Likewise, for every orbit starting at \(n\),
\[
        \max_{\substack{t\ge0\\\Tmap^t(n)\le n+Y}}\tauD(\Tmap^t(n))
        \ge (\log 2-o(1))\frac{\log Y}{(\log\log Y)^2}.
\]
\end{corollary}

\begin{proof}
Let
\[
        G=\left\lfloor (\log 2-\eta)\frac{\log Y}{(\log\log Y)^2}\right\rfloor
\]
with fixed \(\eta>0\).  From \eqref{eq:race-asymp-scale} and \eqref{eq:orbit-asymp-scale}, the corresponding primorial scales are at most \(Y\) for all sufficiently large \(Y\).  The race assertion follows from Theorem \ref{thm:race-quant}; the orbit assertion follows from Theorem \ref{thm:orbit-jumps}.  Finally let \(\eta\downarrow0\).
\end{proof}

\section{Interval filling}\label{sec:interval-filling}

For \(n\in\N\), let \(\mathcal C(n)\) denote the component of \(\Gamma\) containing \(n\).  It is natural to isolate the interval-filling assertion
\begin{equation}\label{eq:IF}
        [n,\Tmap(n)]\cap\N\subseteq\mathcal C(n)\qquad(n\ge1).
\end{equation}
This is not proved unconditionally in this paper.  It is exactly as strong as the original problem.

\begin{proposition}\label{prop:IF-equivalence}
The following are equivalent:
\begin{enumerate}[label=(\alph*)]
\item \(\Gamma\) is connected;
\item \(n\sim n+1\) for every \(n\ge1\);
\item \([n,\Tmap(n)]\cap\N\subseteq\mathcal C(n)\) for every \(n\ge1\).
\end{enumerate}
\end{proposition}

\begin{proof}
The equivalence of (a) and (b) is Lemma \ref{lem:adjacent}.  If \(\Gamma\) is connected, then (c) is immediate.  Conversely, if (c) holds, then \(n+1\in[n,\Tmap(n)]\) because \(\tauD(n)\ge1\), so \(n\sim n+1\) for every \(n\).  Lemma \ref{lem:adjacent} gives connectedness.
\end{proof}

\section{Two-branch criteria and square gates}\label{sec:square-gates}

The previous sections show that static moment estimates alone cannot force global
coalescence.  The next criterion isolates a purely dynamical route: eventual
one-step two-branch collapse, an opposite-parity condition on the two branches,
and infinitely many square-gate levels.

\begin{lemma}[Conditional square reset]\label{lem:square-reset}
Suppose at level \(k\) a two-branch set is \(B_k=\{0,c\}\) with \(c\) odd.  Suppose further that \(B_{k+1}=\AnnT_k(B_k)\) has size at most \(2\), and that whenever it has size \(2\), its two elements have opposite parity.  Then \(|B_{k+1}|=1\).
\end{lemma}

\begin{proof}
By Proposition \ref{prop:parity}, \(\AnnT_k(0)\) is even.  Since \(c>0\) is odd, \(\AnnT_k(c)\equiv c+1\equiv0\pmod2\).  Thus both images have even parity.  Under the stated hypothesis, a two-element image would have opposite parities, which is impossible.  Hence the image is a singleton.
\end{proof}

\begin{theorem}[Square-gated two-branch criterion]\label{thm:square-gated-two-branch}
For \(k\ge2\) put \(U_k=\AnnT_k(E_k)\).  Suppose there is \(K_0\) such that for every \(k\ge K_0\),
\[
        |U_k|\le2,
\]
and whenever \(|U_k|=2\) the two elements of \(U_k\) have opposite parity.  Suppose also that \(0\in U_k\) for arbitrarily large \(k\).  Then \(\Gamma\) is connected.
\end{theorem}

\begin{proof}
Choose arbitrarily large \(k\ge K_0\) with \(0\in U_k\).  If \(|U_k|=1\), then Proposition \ref{prop:dynamic-frontier}, with \(s=1\), gives \(R(k^2-1)\le1\).

It remains to treat the case \(U_k=\{0,c\}\) with \(c\ne0\).  By the opposite-parity hypothesis, \(c\) is odd.  Since \(U_k\subseteq E_{k+1}\), the next image
\[
        \AnnT_{k+1}(U_k)
\]
is a subset of \(U_{k+1}=\AnnT_{k+1}(E_{k+1})\).  Proposition \ref{prop:parity} gives \(\AnnT_{k+1}(0)\) even and \(\AnnT_{k+1}(c)\equiv c+1\equiv0\pmod2\).  Put
\[
        x=\AnnT_{k+1}(0),\qquad y=\AnnT_{k+1}(c).
\]
Both \(x\) and \(y\) belong to \(U_{k+1}\).  If \(x\ne y\), then \(U_{k+1}\) contains two distinct even elements.  Since \(k+1\ge K_0\), the hypothesis gives \(|U_{k+1}|\le2\); hence \(U_{k+1}=\{x,y\}\) would be a two-point set whose two elements have the same parity, contradicting the opposite-parity hypothesis.  Therefore \(x=y\), and \(|\AnnT_{k+1}(U_k)|=1\), i.e. \(|W_{k,2}|=1\).  Proposition \ref{prop:dynamic-frontier} now gives \(R(k^2-1)\le1\).

Since such \(k\) occur arbitrarily far out, every fixed initial interval \([1,X]\) is contained below some \(k^2-1\) with \(R(k^2-1)=1\).  Therefore every integer lies in the component of \(1\).
\end{proof}

\subsection{The Bunyakovsky obstruction to the square-gate route}\label{subsec:bunyakovsky-obstruction}

Theorem \ref{thm:square-gated-two-branch} isolates the extra ingredient missing from a two-branch collapse.  The hypotheses are concrete finite-level conditions on \(U_k=\AnnT_k(E_k)\); they do not assume that two branches have already coalesced.  The bound \(|U_k|\le2\) and the opposite-parity condition alone imply at most two global components by Corollary \ref{cor:dynamic-confluence}.  To collapse two branches to one, the theorem also requires infinitely many square gates \(0\in U_k\).  The next proposition records why that gate is arithmetically much harder than the preceding moment estimates.

\begin{proposition}[Square gates force divisor-value equations]\label{prop:square-gate-divisor-equation}
For \(k\ge2\), put \(U_k=\AnnT_k(E_k)\).  If \(0\in U_k\), then there is an integer \(j\) with \(1\le j\le2k\) such that
\[
        \tauD((k+1)^2-j)=j.
\]
Equivalently, the necessary consequence \(0\in E_{k+1}\) says that the next square frontier has a crossing which lands exactly on \((k+1)^2\).  In particular, an infinite supply of square gates coming from the simplest fixed deficit, namely \(j=2\), would ask for infinitely many primes of the polynomial
\[
        X^2-2.
\]
\end{proposition}

\begin{proof}
By definition, \(U_k\subseteq E_{k+1}\).  If \(0\in U_k\), then \(0\in E_{k+1}\).  Unwinding the definition of \(E_{k+1}\), there is a deficit \(j\) with \(1\le j\le2(k+1)-2=2k\) and
\[
        0=\tauD((k+1)^2-j)-j,
\]
which is the displayed divisor-value equation.  When \(j=2\), this equation is \(\tauD((k+1)^2-2)=2\), which is equivalent to \((k+1)^2-2\) being prime.
\end{proof}

\begin{remark}[Why this is a genuine roadblock]
The proposition gives only a necessary condition for \(0\in U_k\), not a sufficient one: a square landing somewhere in the full annulus gives \(0\in E_{k+1}\), whereas the square-gated criterion needs the dynamically propagated thin frontier \(U_k\) itself to hit that landing.  Thus the square gate contains two layers of difficulty.  First, one must produce divisor-value equations such as \(\tauD((k+1)^2-j)=j\), and already the fixed case \(j=2\) is the Bunyakovsky/Bateman--Horn type problem of primes represented by \(X^2-2\).  Second, one must show that such landings are reached by the relevant branch of the annular dynamics.  Consequently, static divisor bounds for \(|E_k|\), or even density-one two-branch collapse without a square-gate input, cannot by themselves prove Erd\H{o}s--Graham coalescence.  A proof through this route must add genuinely arithmetic information about polynomial divisor values and genuinely dynamical information about which annular landings the thin frontier visits.
\end{remark}

\begin{remark}[What remains genuinely open]\label{rem:dynamic-open-targets}
The much stronger dynamic assertions \(\liminf_k|\AnnT_k(E_k)|=1\), or at least eventual two-branch collapse together with infinitely many square gates \(0\in\AnnT_k(E_k)\), remain open.  The present paper does not prove these statements.  By Corollary \ref{cor:dynamic-confluence}, Theorem \ref{thm:square-gated-two-branch}, and Proposition \ref{prop:square-gate-divisor-equation}, proving them in the stated strength would settle the Erd\H{o}s--Graham coalescence problem but would also require input beyond the static moment estimates above.  They should therefore be regarded as the central remaining analytic-dynamical obstacles.
\end{remark}

\section*{Acknowledgements}
The author acknowledges the use of OpenAI's ChatGPT during the preparation of this manuscript. While it was used for ideation, formulation, proof exploration and refinement, narrowing the search space, LaTeX formatting and other forms of orchestration, the author nonetheless takes full responsibility for the accuracy of the final contents of this paper.

\end{document}